\documentclass[11pt]{article}
\usepackage{preprint-lt}
\usepackage{appendix}
\usepackage{geometry}
\usepackage{amssymb}
\usepackage{amsmath}
\usepackage{amsthm}
\usepackage{mathtools}
\usepackage{mathrsfs}
\DeclareFontFamily{U}{rsfs}{\skewchar\font127 }
\DeclareFontShape{U}{rsfs}{m}{n}{%
   <-6> rsfs5
   <6-8> rsfs7
   <8-> rsfs10
}{}
\usepackage{hyperref}
\hypersetup{
	colorlinks=true,
	linkcolor=cyan!50!blue,
	filecolor=blue,      
	urlcolor=red,
	citecolor=green,
}
\usepackage{cleveref}
\usepackage[numbers,sort]{natbib}
\usepackage{bm}
\usepackage{framed}
\usepackage{graphicx}
\usepackage{xcolor}
\usepackage{listings}
\usepackage{multirow}
\usepackage{multicol}
\usepackage{indentfirst}
\usepackage{booktabs}
\usepackage{tikz}
\usepackage[all]{xy}
\usepackage{float}
\usepackage{algorithm}
\usepackage{algorithmicx}
\usepackage{algpseudocode}
\usepackage{subfigure}
\usepackage{lineno}
\newcommand*{\be}[1]{\begin{equation}\label{#1}}
\newcommand*{\ee}{\end{equation}}

\makeatletter
\@addtoreset{equation}{section}
\makeatother

\DeclareMathOperator{\grad}{grad}

\DeclareMathOperator{\curl}{curl}
\DeclareMathOperator{\sym}{sym}
\DeclareMathOperator{\diverenge}{div}
\DeclareMathOperator{\im}{im}
\renewcommand{\div}{\diverenge}
\newcommand{\bR}{\mathbb{R}}
\newcommand{\cP}{\mathcal{P}}
\newcommand{\cQ}{\mathcal{Q}}
\newcommand{\bS}{\mathbb{S}}
\newcommand{\bT}{\mathbb{T}}

\newcommand{\bM}{\mathbb{M}}

\newcommand{\cT}{\mathcal{T}}
\newcommand{\px}{\frac{\partial}{\partial x}}
\newcommand{\py}{\frac{\partial}{\partial y}}
\newcommand{\pz}{\frac{\partial}{\partial z}}
\newcommand{\pxy}{\frac{\partial^2}{\partial x \partial y}}
\newcommand{\pxz}{\frac{\partial^2}{\partial x \partial z}}
\newcommand{\pyz}{\frac{\partial^2}{\partial y \partial z}}
\newcommand{\pxyz}{\frac{\partial^3}{\partial x \partial y \partial z}}

\DeclareMathOperator{\dev}{dev}

\newcommand{\for}{\text{ for }}

\DeclareMathOperator{\Span}{span}
\title{Finite element grad grad complexes and elasticity complexes on cuboid meshes}
\author{\name Jun Hu \inst Peking University \email hujun@math.pku.edu.cn \and
	\name Yizhou Liang \inst Peking University \email lyz2015@pku.edu.cn
	\and
	\name  Ting Lin \inst Peking University \email lintingsms@pku.edu.cn}

\begin{document}

\maketitle
% \begin{abstract}
%     In this paper, a family of discrete grad-grad complexes on three-dimensional rect angular grids is constructed. To achieve this, a set of scalar $H^2$, symmetric tensor $H(\curl)$ and traceless $H(\div)$ conforming finite element spaces is proposed, ensuring that the discrete complex is exact. Utilizing the proposed elements, a finite element discretization is given to solve linearized Einstein--Bianchi equation.
% \end{abstract}
\begin{abstract}
This paper constructs two conforming finite element grad grad and elasticity complexes on the cuboid meshes. For the finite element grad grad complex, an $H^2$ conforming finite element space, an $\bm H(\curl; \mathbb S)$ conforming finite element space, an $\bm H(\operatorname{div}; \mathbb T)$ conforming finite element space and an $\bm L^2$ finite element space are constructed. Further, a finite element complex with reduced regularity is also constructed, whose degrees of freedom for the three diagonal components are coupled. For the finite element elasticity complex, a vector $\bm H^1$ conforming space and an $\bm H(\curl \curl^{\mathsf T}; \mathbb S)$ conforming space are constructed. Combining with an existing $\bm H(\div;\mathbb S) \cap \bm H(\div\div;\mathbb S)$ element and $\bm H(\div; \mathbb S)$ element, respectively, these finite element spaces form two different finite element elasticity complexes. The exactness of all the finite element complexes is proved.
\end{abstract}
\tableofcontents

 \section{Introduction}
Differential complexes have been an important tool in the study and design of finite element methods~\cite{2006ArnoldFalkWinther,2010ArnoldFalkWinther,2018Arnold}. The most canonical differential complex is the de Rham complex, and it plays an important role in the finite element study of electromagnetism and fluid dynamics. In this paper, we focus on the construction of finite elements for another two differential complexes, i.e., the so-called gradgrad complex \cite{dirk2020,2021ArnoldHu}
\begin{equation}\label{eq:intro:gradgrad}
      \cP_1 \stackrel{\subset}{\longrightarrow} H^2(\Omega)\stackrel{\grad \grad}{\longrightarrow} \bm H(\curl,\Omega; \bS) \stackrel{\curl}{\longrightarrow} \bm H(\div,\Omega; \bT) \stackrel{\div}{\longrightarrow} \bm L^2(\Omega) \longrightarrow 0,
    \end{equation}
and the elasticity complex \cite{2008ArnoldAwanouWinther,2021ArnoldHu} 
    \begin{equation}\label{eq:intro:elasticity}
        \bm{\mathcal{RM}} \stackrel{\subseteq}{\longrightarrow} \bm H^{1}\left(\Omega\right) \stackrel{\operatorname{sym} \operatorname{grad}}{\longrightarrow} \bm H(\operatorname{curl} \operatorname{curl}^{\mathsf{T}}, \Omega ; \mathbb{S}) \stackrel{\operatorname{curl} \operatorname{curl}^{\mathsf{T}}}{\longrightarrow} \bm H(\operatorname{div}, \Omega ; \mathbb{S}) \stackrel{\div}{\longrightarrow} \bm L^{2}(\Omega) \longrightarrow 0.
        \end{equation}
Here $\mathbb{T}$ and $\mathbb{S}$ denote the spaces of traceless and symmetric matrices in three dimensions,
respectively, the operators $\operatorname{curl}$ and $\operatorname{div}$ act row-wise on the matrix-valued functions, and the operator $\operatorname{curl}^{\mathsf T}$ acts column-wise on the matrix-valued functions. Here $\mathcal P_1$ is the space of linear function and the rigid motion space $\bm{\mathcal{RM}} := \{\bm a + \bm b \times \bm x:\bm a, \bm b \in \mathbb R^3\}.$

The finite element discretization of the gradgrad complex \eqref{eq:intro:gradgrad} is related to the linearized Einstein-Bianchi system~\cite{quenneville2015new}. The first finite element
 gradgrad complex on tetrahedral grids was constructed in \cite{2021HuLiang}, and those finite element spaces can be used to solve the linearized Einstein-Bianchi system within the mixed form. Recently, several discrete divdiv complexes were constructed \cite{hu2021divdiv1,hu2022new,2020ChenHuang3D}, and the associated finite element
spaces can be used to discretize the linearized Einstein-Bianchi system within the
dual formulation introduced in \cite{quenneville2015new}. 

The elasticity complex~\eqref{eq:intro:elasticity} plays an important role in the theoretical and numerical analysis of linear elasticity problems, cf.~\cite{2002ArnoldWinther,2008ArnoldAwanouWinther}. It can be derived from the composition of de Rham complexes in the so-called Bernstein-Gelfand-Gelfand (BGG)
construction \cite{2021ArnoldHu}. By the BGG construction, a two-dimensional finite element elasticity complex has been constructed in \cite{2018ChristianseHuHu}. On the Clough-Tocher split in two dimensions, a finite element elasticity complex was proposed in \cite{christiansen2022finite};  On the Alfeld split in three dimensions, a finite element elasticity complex was proposed in \cite{2020Christiansen}. Recently, a finite element elasticity complex on a general tetrahedral mesh was constructed in \cite{chen2022finite} in which the $H(\operatorname{div},\Omega;\mathbb{S})$ finite element is the Hu-Zhang element for the symmetric stress tensor~\cite{MR3352360,MR3301063,MR3529252}.

In this paper, we construct new families of finite element complexes of these two complexes on cuboid meshes. For each complex, two families of discrete complexes with different local and global regularity have been constructed. For the finite element gradgrad complex, the $H^2$ finite element is the three-dimensional Bogner--Fox--Schmit (BFS) element, whose restriction on each face of each element is a two-dimensional BFS element. The $\bm H(\curl;\mathbb S)$ conforming finite element space constructed in this paper is shown to possess higher regularity. In fact, the finite element space is also $\bm H(\curl \curl^{\mathsf{T}}; \mathbb S)$ conforming, and each component is $H^1$ conforming. As a result, this finite element is also used in the construction of the finite element elasticity complexes. 
For the $\bm H(\div; \mathbb T)$ conforming finite element space, the degrees of freedom defined on each diagonal component are actually those of the Lagrange element. As a result, the diagonal components admit $H^1$ regularity. The regularity seems necessary if the degrees of freedom on each component are considered separately. 
% The degrees of freedom of the off-diagonal components are proposed to make the whole matrix-valued element unisolvent and $H(\div; \mathbb T)$ conforming.
The finite element complexes formed by the above spaces are shown to be exact in a contractible domain. Further, an $\bm H(\curl; \mathbb S)$ element and an $\bm H(\div; \mathbb T)$ element with reduced regularity are constructed. In this case, the $\bm H(\curl; \mathbb S)$ element is not $\bm H(\curl \curl^{\mathsf T};\mathbb S)$ conforming, and the degrees of freedom for diagonal components of $\bm H(\div; \mathbb T)$ element are coupled. For the finite element elasticity complex, the $\bm H^1$ conforming element is newly constructed, and the $\bm H(\curl \curl^{\mathsf T}; \mathbb S)$ element comes from the construction of the above finite element gradgrad complex. There are two different choices of $\bm H(\div;\mathbb S)$ elements, one is $\bm H(\div;\mathbb S)$ conforming element from in \cite{hu2014simple}, the other is the $\bm H(\div;\mathbb S) \cap \bm H(\div\div;\mathbb S)$ element from \cite{hu2022new}. Combined with the cuboid Brezzi--Douglas--Marini \cite{2013BoffiBrezziFortin} and DG element, respectively, two finite element elasticity complexes are formed and proved to be exact.

The rest of the paper is organized as follows: Section 2 introduces the notation. Section 3 constructs two families of finite element gradgrad complexes with different regularity on cuboid meshes. Section 4 designs two families of the finite element elasticity complexes.  

\section{Notations}

Let $\Omega$ be a contractible domain with Lipschitz boundary, which can be partitioned into cuboid grids.
Denote by $\mathbb{M}$ the space of $(3\times 3)$ matrices, and $\bS$ and  $\bT$ the space of symmetric and traceless matrices, respectively. For convenience, the components of vectors and matrices are all indexed by $x$, $y$, and $z$. For example,
\begin{equation}
\bm v = \begin{bmatrix} v_x \\ v_y  \\ v_z \end{bmatrix}
\end{equation}
and
\begin{equation}
\bm \sigma = \begin{bmatrix} \sigma_{xx} & \sigma_{xy} & \sigma_{xz} \\ 
    \sigma_{yx} & \sigma_{yy} & \sigma_{yz} \\
    \sigma_{zx} & \sigma_{zy} & \sigma_{zz} \end{bmatrix}.
\end{equation}

For matrices, the operators $\curl$ and $\div$ are applied on each row, while $\curl^{\mathsf T}$ means applying the curl operator on each column, that is, $\curl^{\mathsf T} \sigma = (\curl \sigma^{\mathsf T})^{\mathsf T}$ for $\sigma \in \mathbb M$. The symmetric part of matrix $\sigma$ is denoted as $\sym \sigma = \frac{1}{2}(\sigma+\sigma^{\mathsf T})$, and the traceless part is denoted as $\dev \sigma = \sigma - \frac{1}{3}(\sigma_{xx} + \sigma_{yy} + \sigma_{zz}).$

The mesh considered in this paper will be a cuboid mesh $\mathcal T_h$. For edge $e$ of an element $K$, the subscript will be imposed to indicate its direction. For example, $e_x$ represents an edge parallel to the $x$-direction. Similar notations will be adopted for the faces of elements, for example, $F_{yz}$ represents a face normal to the $x$-direction, see \Cref{fig:cuboid}.

\begin{figure}
    \centering
    \includegraphics{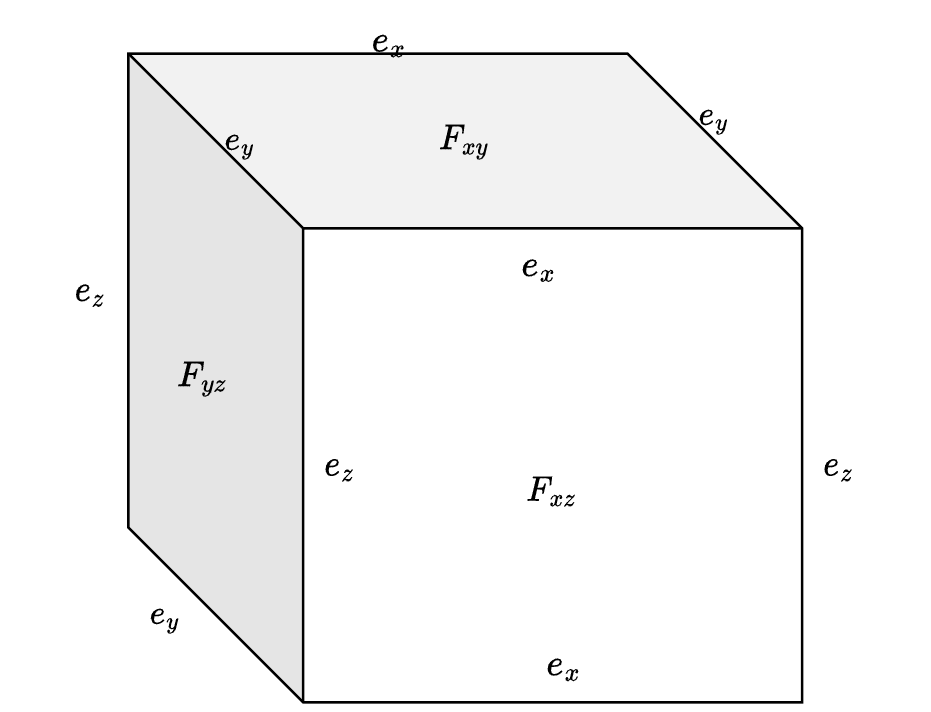}
    \caption{An illustration of notations.}
    \label{fig:cuboid}
\end{figure}
Set 
\begin{equation}
\cQ_{k_1,k_2,k_3}(x,y,z) := \Span \{ x^{l_1}y^{l_2}z^{l_3} : 0 \le l_i \le k_i \}.
\end{equation}
Unless otherwise specified, $\cQ_{k_1,k_2,k_3}(x,y,z)$ will be shorten as $\cQ_{k_1,k_2,k_3}$, and the polynomial spaces on edges and faces $\cQ_{k}(x)$ and $\cQ_{k_1,k_2}(x,y)$ are defined similarly. The corresponding space is defined as a null space when one of the index $k_i$ is negative.

Two finite element spaces on the cuboid mesh will be used without extra definition: the Lagrange finite element space
$$\mathcal{L}_{k,k,k} = \{ u \in C^0(\Omega); u|_K \in \cQ_{k,k,k}, K \in \mathcal T_h\},$$ and the discontinuous finite element space $$\mathcal{DG}_{k_1,k_2,k_3} = \{ u \in L^2(\Omega); u|_K \in \cQ_{k_1,k_2,k_3}, K \in \mathcal T_h\}.$$

For unisolvency, the proof is considered on the reference element $T = [0,1]^3$. To prove the exactness, it is necessary to count the dimension of the constructed finite element spaces. Denote by $\mathscr V$ the number of vertices, $\mathcal E$ the number of edges, $\mathscr F$ the number of faces, and $\mathscr T$ the number of cells. 

Henceforth, the vector-valued Sobolev spaces are given by $\bm H^1$ and $\bm L^2$.
 \section{Discrete gradgrad complex}
\label{sec:fe}
This section considers four types of finite element spaces: $H^2$ conforming space $U_h$, $\bm H(\curl, \Omega;\bS)$ conforming space $\bm{\Sigma}_h$, $\bm H(\div, \Omega;\bT)$ conforming space $\bm{\Xi}_h$ and $\bm L^2(\Omega)$ space $\bm{Q}_h$. These finite element spaces will be used to construct 
\begin{equation}
    % \label{eq:complex-gradgrad-D}
    \cP_1 \stackrel{\subset}{\longrightarrow} U_h \stackrel{\grad \grad}{\longrightarrow} \bm \Sigma_h \stackrel{\curl}{\longrightarrow} \bm \Xi_h \stackrel{\div}{\longrightarrow} \bm Q_h \longrightarrow 0,
\end{equation}
a discrete subcomplex of the following continuous gradgrad complex,
\begin{equation}
    \label{eq:complex-gradgrad-C}
      \cP_1 \stackrel{\subset}{\longrightarrow} H^2(\Omega)\stackrel{\grad \grad}{\longrightarrow} \bm H(\curl,\Omega; \bS) \stackrel{\curl}{\longrightarrow} \bm H(\div,\Omega; \bT) \stackrel{\div}{\longrightarrow} \bm L^2(\Omega) \longrightarrow 0,
    \end{equation}
    where $\cP_1$ is the space of polynomials of degree $\le 1$ ($\dim = 4$), the spaces

    \begin{equation}
        \bm H(\curl,\Omega;\bS) := \{v \in \bm L^2(\Omega; \bS) ~|~ \curl v \in \bm L^2(\Omega; \bM) \},
    \end{equation}

    and 
    \begin{equation}
        \bm H(\div,\Omega;\bT) := \{v \in \bm L^2(\Omega; \bT) ~|~ \div v \in \bm L^2(\Omega, \bR^3) \}.
    \end{equation}

\subsection{Local version: the polynomial complex}
\label{sec:gradgrad:local}
In this subsection, the local version of the finite element subcomplex is constructed with the following form:
\begin{equation}\label{eq:complex-gradgrad-poly}
    \bm \cP_1 \stackrel{\subset}{\longrightarrow} U_T \stackrel{\grad \grad}{\longrightarrow} \bm \Sigma_T \stackrel{\curl}{\longrightarrow} \bm \Xi_T \stackrel{\div}{\longrightarrow} \bm Q_T \longrightarrow 0,
\end{equation}
for $k \ge 3$, where the local spaces
$$ U_T := \cQ_{k,k,k},$$ 
$$ \bm \Sigma_T := \begin{bmatrix}
    \cQ_{k-2,k,k} & \cQ_{k-1,k-1,k} & \cQ_{k-1,k,k-1}\\
    \cQ_{k-1,k-1,k} & \cQ_{k,k-2,k} & \cQ_{k,k-1,k-1} \\
    \cQ_{k-1,k,k-1} & \cQ_{k,k-1,k-1} & \cQ_{k,k,k-2}
\end{bmatrix},$$
$$
\bm \Xi_T := \begin{bmatrix}
    \cQ_{k-1,k-1,k-1} & \cQ_{k-2,k,k-1} &\cQ_{k-2,k-1,k} \\
    \cQ_{k,k-2,k-1} & \cQ_{k-1,k-1,k-1} & \cQ_{k-1,k-2,k} \\
    \cQ_{k,k-1,k-2} & \cQ_{k-1,k,k-2} &\cQ_{k-1,k-1,k-1}
\end{bmatrix} ,$$
and 
$$\bm Q_T := \begin{bmatrix}
    \cQ_{k-2,k-1,k-1} \\ \cQ_{k-1,k-2,k-1} \\\cQ_{k-1,k-1,k-2} \\ 
\end{bmatrix}.$$

Clearly, \eqref{eq:complex-gradgrad-poly} is a complex for the choice of local function spaces, the dimension counting 
is currently admitted and this will be postponed to the global discrete complex. 
\begin{proposition}
The polynomial sequence \eqref{eq:complex-gradgrad-poly} is an exact complex.
\end{proposition}
\begin{proof}
It suffices to show the discrete divergence operator is surjective, and the kernel of the discrete $\curl$ operator is the image of $\grad \grad$.

First, for $\bm q \in Q_T$, define $\bm v \in \bm \Xi_T$ such that 
$v_{xy}(x,y,z) = \int_{0}^{ y} q_x ds$, and similarly define $v_{yz}$, $v_{zx}$. The other components of $\bm v$ are set as zero. Then it holds that $\div \bm v = \bm q$, which proves $\div \bm \Xi_{T} = \bm Q_T$.

Second, suppose that $\bm \sigma \in \Sigma_T$ such that $\curl \bm \sigma = 0$, then there exists $\bm \phi = [\phi_x, \phi_y,\phi_z] \in [\cQ_{k-1,k,k}, \cQ_{k,k-1,k}, \cQ_{k,k,k-1}]$, such that $\bm \sigma = \grad \bm \phi$. The symmetric property of matrix $\bm \sigma$ implies that $\py \phi_x = \px \phi_y$. Together with the other off-diagonal entries, this leads to that $\curl \bm \phi = 0$, which implies there exists $u \in \cQ_{k,k,k}$ such that $\bm \phi = \grad u$. Hence, it holds that $\bm \sigma = \grad \grad u.$ This is, $\grad \grad U_{T} = \{\bm \sigma \in \bm \Sigma_{T} : \curl \bm \sigma = 0\}.$
A combination of these two facts and the dimension counting indicates the exactness.
\end{proof} 
The above proposition, together its intermediate results, is useful to prove the exactness of the global finite element complexes.

\subsection{\texorpdfstring{$H^2$}{H2} conforming finite element space}
 The $H^2$ conforming finite element space $U_h$ is a generalization of Bogner--Fox--Schmit (BFS) element spaces \cite{schmit1968finite,brenner2007mathematical} in three dimensions, where the shape function space is $\cQ_{k,k,k}$ for $ k \ge 3$. Given $u \in \cQ_{k,k,k}$, the degrees of freedom are defined as follows:

\begin{enumerate}
\item The function value and partial derivatives of $u$ at each vertex $\bm x$ of $T$, \begin{equation}
\label{eq:dof-H2-v}
u(\bm{x}), \px u(\bm{x}), \py u(\bm{x}), \pz u(\bm{x}), \pxy u(\bm{x}), \pxz u(\bm{x}), \pyz u(\bm{x}), \pxyz u(\bm{x}).
\end{equation}
\item
The moments of $u$ and its normal derivatives on each edge $e$, say $e_x$, of $T$,
\begin{equation}
    \label{eq:dof-H2-ex} 
    \int_{e_x} u p ~dx,~~~ \int_{e_x} \py u p ~dx,~~~ \int_{e_x} \pz u p ~dx,~~~ \int_{e_x} \pyz u p ~dx \quad \text{ for } p \in \cQ_{k-4}(x),
\end{equation} 
and the degrees of freedom defined on $e_y$ and $e_z$ are similarly defined, by a cyclic permutation. \footnote{This means  $x \mapsto y, y \mapsto z, z \mapsto x$.}
% The degrees of freedom by cyclic permutation ($x \mapsto y, y\mapsto z, z\mapsto x$).

\item The moments of $u$ and its normal derivatives on each face $F$, say $F_{yz}$, of $T$,
\begin{equation}
    \label{eq:dof-H2-fyz}
     \int_{F_{yz}} u p ~dydz,~~~ \int_{F_{yz}} \px u p ~dydz \quad \text{ for } p \in \cQ_{k-4,k-4}(y,z),
\end{equation} 
and the degrees of freedom defined on the other faces $F_{xz}$ and $F_{xy}$, by a cyclic permutation.
% The degrees of freedom on $F_{xz}$, $F_{xy}$ by cyclic permutation ($x \mapsto y, y\mapsto z, z\mapsto x$).

\item The moments of $u$ inside the element $T$, \begin{equation}\label{eq:dof-H2-t}
    \int_{T} up~dxdydz \quad \text{ for } p \in \cQ_{k-4,k-4,k-4}(x,y,z).
\end{equation}

\end{enumerate}

It will be shown that the set of degrees of freedom is unisolvent with respect to the shape function space $\cQ_{k,k,k}$, and that the resulting finite element space is a subspace of $H^2(\Omega)$.
\begin{proposition}
Suppose $k\ge 3$, then the above set of degrees of freedom is unisolvent with respect to the shape function space $\cQ_{k,k,k}$, and the finite element space $U_h$ is $H^2$ conforming.
\end{proposition}
\begin{proof}
The dimension of the shape function space is $(k+1)^3$, which is equal to the number of the total degrees of freedom is $(k-3)^3 + 12\times(k-3)^2 + 36\times(k-3)+64 = (k+1)^3$. 
Hence it suffices to prove that if $u \in \cQ_{k,k,k}$ vanishes at all degrees of freedom \eqref{eq:dof-H2-v}, \eqref{eq:dof-H2-ex}, \eqref{eq:dof-H2-fyz}, \eqref{eq:dof-H2-t}, then $u = 0$. Since for each face $F$ with the normal vector $\bm n$, $u|_F, \frac{\partial u}{\partial \bm n}|_F \in \cQ_{k,k}$ vanishes for all the degrees of freedom of the two-dimensional BFS element. Hence, it follows from the unisolvency of the two-dimensional BFS element that $u = \frac{\partial u}{\partial \bm n} = 0$ on the faces of the element $T$.

Therefore, on $T = [0,1]^3$, the standard argument yields that $u = x^2(1-x)^2y^2(1-y)^2z^2(1-z)^2 u_1$ for some $u_1 \in \cQ_{k-4,k-4,k-4}(x,y,z)$. The degrees of freedom inside element $T$ make sure that $u_1 = 0$. This proves the unisolvency, while the $H^2$ continuity is implied by the previous argument.

\end{proof}
Note that the dimension of the space $U_h$ is as follows:
\begin{equation}\dim U_h = 8\mathscr{V} + 4(k-3)\mathscr{E} + 2(k-3)^2\mathscr{F} + (k-3)^3\mathscr{T}.
\end{equation}

\subsection{\texorpdfstring{$\bm H(\curl;\bS)$}{H(curl;S)} conforming finite element space}
\label{sec:HcurlS}
This subsection considers the construction of an $\bm H(\curl;\bS)$ conforming finite element space $\bm{\Sigma}_h$ on $\mathcal T_h$.  
For this element, the shape function space on $T$ is taken as

\begin{equation}
\label{eq:shapefunc-HcurlS}
\begin{bmatrix}
\sigma_{xx} & \sigma_{xy} & \sigma_{xz} \\
\sigma_{yx} & \sigma_{yy} & \sigma_{yz} \\
\sigma_{zx} & \sigma_{zy} & \sigma_{zz}
\end{bmatrix}
\in
\begin{bmatrix}
    \cQ_{k-2,k,k} & \cQ_{k-1,k-1,k} & \cQ_{k-1,k,k-1}\\
    \cQ_{k-1,k-1,k} & \cQ_{k,k-2,k} & \cQ_{k,k-1,k-1} \\
    \cQ_{k-1,k,k-1} & \cQ_{k,k-1,k-1} & \cQ_{k,k,k-2}
\end{bmatrix} =: \bm \Sigma_T.
\end{equation}

The following lemma is useful to prove the $\bm H(\curl; \bS)$ conformity of a symmetric matrix-valued piecewise polynomial on the cuboid mesh $\mathcal T_h$. 
\begin{lemma}[A sufficient condition for $\bm H(\curl; \bS)$ conformity]
    \label{lem:HcurlS-conformity}
    If a symmetric matrix-valued piecewise polynomial $\bm \sigma \in \bS$ satisfies that 
    \begin{enumerate}
        \item $\sigma_{xx}$ is single-valued across all the faces $F_{xy}$ and $F_{xz}$ of $\mathcal T_h$, and a corresponding  condition holds for $\sigma_{yy}$ and $\sigma_{zz}$, with the index changing cyclicly.
        \item All the off-diagonal components, say $\sigma_{xy}$, are single-valued across all the faces of $\mathcal T_h$. 
    \end{enumerate}
    Then $\bm \sigma$ is in $\bm H(\curl, \Omega; \bS)$. 
\end{lemma}
\begin{proof}

It follows from a straightforward argument. 
\end{proof}

For $\bm \sigma \in \bm{\Sigma}_h$, due to its symmetry, the degrees of freedom will be separated into six parts as follows:$$\sigma_{xx}, \sigma_{yy}, \sigma_{zz}, \sigma_{xy} = \sigma_{yx}, \sigma_{xz} = \sigma_{zx}, \sigma_{yz} = \sigma_{zy}.$$ In what follows, only the degrees of freedom of $\sigma_{xx}$ and $\sigma_{xy}$ will be specified, since those of the remaining four components can be similarly defined via cyclic permutation.

\textbf{The degrees of freedom of $\sigma_{xx}$} are defined as follows:
\begin{enumerate}
    \item The moments of $\sigma_{xx}$ and its  normal derivatives on each edge $e_x$ of $T$, \begin{equation}\label{eq:dof-HcurlS-e}
    \int_{e_x} \sigma_{xx}p~dx, \int_{e_x} \py \sigma_{xx}p~dx, \int_{e_x} \pz \sigma_{xx}p~dx,
    \int_{e_x} \pyz \sigma_{xx}p~dx \text{ for }p \in \cQ_{k-2}(x).
    \end{equation}
    \item The moments of $\sigma_{xx}$ and its normal derivative on each face $F_{xy}$ of $T$,
    \begin{equation}\label{eq:dof-HcurlS-fxy}
    \int_{F_{xy}} \sigma_{xx}p~dxdy,\int_{F_{xy}}\pz \sigma_{xx}p~dxdy \for p \in \cQ_{k-2,k-4}(x,y).
    \end{equation}
    \item The moments of $\sigma_{xx}$ and its normal derivative on each face $F_{xz}$ of $T$,
    \begin{equation}\label{eq:dof-HcurlS-fxz}
         \int_{F_{xz}} \sigma_{xx}p~dxdz,\int_{F_{xz}}\py \sigma_{xx}p~dxdz  \for p \in \cQ_{k-2,k-4}(x,z).\end{equation}
    \item The moments of $\sigma_{xx}$ inside $T$, \begin{equation}\label{eq:dof-HcurlS-t} \int_{T} \sigma_{xx}p ~dxdydz \for p \in \cQ_{k-2,k-4,k-4}(x,y,z).\end{equation}

\end{enumerate}
% When $k = 3$, \eqref{eq:dof-HcurlS-fxy}, \eqref{eq:dof-HcurlS-fxz} and \eqref{eq:dof-HcurlS-t} disappear. 
The degrees of freedom of $\sigma_{yy}, \sigma_{zz}$ can be defined in a similar way as those of $\sigma_{xx}$ by cyclic permutation.

\textbf{The degrees of freedom of $\sigma_{xy}$} are defined as follows:
\begin{enumerate}
    \item The value of $\sigma_{xy}$ and its partial derivative along $z$-direction, at each vertex $\bm x$ of $T$, \begin{equation}\label{eq:dof-HcurlS2-v}\sigma_{xy}(\bm{x}),  \pz\sigma_{xy}(\bm{x}).\end{equation}
    \item The moments of $\sigma_{xy}$ and its normal derivative on each $e_x$ of $T$,
    \begin{equation}\label{eq:dof-HcurlS2-ex}
    \int_{e_x} \sigma_{xy} p ~dx, \int_{e_x} \pz\sigma_{xy} p~dx \for p \in \cQ_{k-3}(x).
    \end{equation}
    \item The moments of $\sigma_{xy}$ and its normal derivative on each $e_y$ of $T$,
     \begin{equation}\label{eq:dof-HcurlS2-ey}\int_{e_y} \sigma_{xy} p ~dy, \int_{e_y} \pz\sigma_{xy} p~dy \for  p \in \cQ_{k-3}(y).\end{equation}
     \item The moments of $\sigma_{xy}$ on each $e_z$ of $T$,
     \begin{equation}\label{eq:dof-HcurlS2-ez}\int_{e_z} \sigma_{xy} p~dx \for  p \in \cQ_{k-4}(z).\end{equation}
     
     \item The moments of $\sigma_{xy}$ and its normal derivatives on each face $F_{xy}$ of $T$,
     \begin{equation}\label{eq:dof-HcurlS2-fxy}\int_{F_{xy}} \sigma_{xy} p ~dxdy, \int_{F_{xy}} \pz \sigma_{xy} p ~dxdy \for p \in \cQ_{k-3,k-3}(x,y).\end{equation}
     \item The moments of $\sigma_{xy}$ on each face $F_{xz}$ of $T$, \begin{equation}\label{eq:dof-HcurlS2-fxz}\int_{F_{xz}} \sigma_{xy} p ~dxdz \for p \in \cQ_{k-3,k-4}(x,z).\end{equation} 
     \item The moments of  $\sigma_{xy}$ on each face $F_{yz}$ of $T$, 
     \begin{equation}\label{eq:dof-HcurlS2-fyz}\int_{F_{yz}} \sigma_{xy} p~dydz \for p \in \cQ_{k-3,k-4}(y,z).\end{equation}
     \item The moments of $\sigma_{xy}$ inside the element $T$, \begin{equation}\label{eq:dof-HcurlS2-t}\int_T \sigma_{xy} p ~dxdydz \for p \in \cQ_{k-3,k-3,k-4}(x,y,z).\end{equation}
\end{enumerate}
% When $k = 3$, \eqref{eq:dof-HcurlS2-ez}, \eqref{eq:dof-HcurlS2-fxz},\eqref{eq:dof-HcurlS2-fyz} and \eqref{eq:dof-HcurlS2-t} disappear. 
The degrees of freedom of $\sigma_{yz}$ and $\sigma_{zx}$ are similar those of $\sigma_{xy}$, by cyclic permutation. 
 
The following proposition shows that the set of degrees of freedom is unisolvent and the resulting finite element space is $\bm H(\curl;\bS)$  conforming.
\begin{proposition}
\label{prop:uni-HcurlS}
When $k \ge 3$, the degrees of freedom defined in \eqref{eq:dof-HcurlS-e}-\eqref{eq:dof-HcurlS2-t} are unisolvent for the shape function space $\bm \Sigma_T$ in \eqref{eq:shapefunc-HcurlS}, and the resulting finite element space $\bm \Sigma_h$ is $\bm H(\curl; \bS)$ conforming.
\end{proposition}
\begin{proof}
% A sketch of proof will be given first.  To prove the unisolvency, it suffices to prove each coordinate is unisolvent. By symmetry, it suffices to consider $\sigma_{xx}$ and $\sigma_{xy}$, and others are similar. Second, to prove $\bm H(\curl)$ conforming, it suffices to show $\bm{\sigma} \times \bm{n}$ is continuous, where $\bm{n}$ is the normal vector. Again by symmetry, it is easy to see that it remains to show $\sigma_{xx}$ are continuous when crossing the face $F_{xy},F_{xz}$ and $\sigma_{xy}$ is continuous when crossing each face. 
The unisolvency of diagonal entries and off-diagonal entries will be separately considered. It suffices to show the unisolvency of $\sigma_{xx}$ and $\sigma_{xy}$. 
\paragraph{Unisolvency of $\sigma_{xx}$}
The number of degrees of freedom defined for $\sigma_{xx}$ is 
\begin{equation}
    (k-1)(k-3)^2 + 8(k-1)(k-3) + 16(k-1) = (k-1)(k-3+4)^2 = (k-1)(k+1)^2,
\end{equation}
which is equal to the dimension of $\cQ_{k-2,k,k}$.
It remains to show that if $\sigma_{xx} \in \cQ_{k-2,k,k}$ vanishes for these degrees of freedom, then $\sigma_{xx} = 0$. 
Since 
$$\displaystyle \sigma_{xx}|_{e_x},\frac{\partial \sigma_{xx}}{\partial y}|_{e_x},\frac{\partial \sigma_{xx}}{\partial z}|_{e_x},\frac{\partial^2 \sigma_{xx}}{\partial y\partial z}|_{e_x} \in \cQ_{k-2}(x),$$ 
the degrees of freedom in \eqref{eq:dof-HcurlS-e} imply that they vanish on $e_{x}$. When $k = 3$, consider any face, say $F_x$, normal to $x$-direction. 
The restriction $\sigma_{xx}|_{F_x}$ of $\sigma_{xx}$ on such a face is in $\cQ_{k,k}(y,z)$, 
and vanishes at each degree of freedom of the two-dimensional BFS element, hence it can be concluded that $\sigma_{xx} = 0$.

Now consider the case $k \ge 4$. Since $\sigma_{xx}|_{F_{xy}}$ and $\displaystyle \pz \sigma_{xx}|_{F_{xy}}$ are in $\cQ_{k-2,k-4}$, the degrees of freedom in \eqref{eq:dof-HcurlS-fxy} show that $\sigma_{xx},\frac{\partial \sigma_{xx}}{\partial z}$ vanish on $F_{xy}$. A similar argument shows that $u,\frac{\partial \sigma_{xx}}{\partial y}$ vanish on $F_{xz}$.
% , this result implies that $\sigma_{xx}$ and $\frac{\partial \sigma_{xx}}{\partial \bm n}$ are single-valued on the faces $F_{xy},F_{xz}$.
It follows that $\sigma_{xx} = y^2(1-y)^2z^2(1-z)^2\sigma_1$ with $\sigma_1\in \cQ_{k-2,k-4,k-4}$. Finally, the degrees of freedom in \eqref{eq:dof-HcurlS-t} imply that $\sigma_{xx}=0$.

\paragraph{Unisolvency of $\sigma_{xy}$}
Now it turns to show the unisolvency of $\sigma_{xy}$. The number of the degrees of freedom in \eqref{eq:dof-HcurlS2-v}--\eqref{eq:dof-HcurlS2-t} is
\begin{equation}
    16 + 16(k-2) + 4(k-3) + 4(k-2)^2 + 4(k-2)(k-3) + (k-2)^2(k-3) = k^2(k+1),
\end{equation}
which is equal to the dimension of the space $\cQ_{k-1,k-1,k}$. 
It suffices to show that for any polynomial $ \sigma_{xy} \in \cQ_{k-1,k-1,k}$, it vanishes at all the degrees of freedom if and only if $\sigma_{xy} =0$ on element $T$. It follows from the degrees of freedom defined in \eqref{eq:dof-HcurlS2-v}, \eqref{eq:dof-HcurlS2-ex}, \eqref{eq:dof-HcurlS2-ey} that $\sigma_{xy}$ and $\frac{\partial \sigma_{xy}}{\partial z}$ vanish on all $e_{x},e_{y}$. 
When $k \ge 4$, it follows from \eqref{eq:dof-HcurlS2-v} and \eqref{eq:dof-HcurlS2-ez} that $\sigma_{xy}$ vanishes on all the edges $e_{z}$. It then follows from \eqref{eq:dof-HcurlS2-fxy} that $\sigma_{xy}$ and $\frac{\partial \sigma_{xy}}{\partial z}$ vanish on face $F_{xy}$, and from \eqref{eq:dof-HcurlS2-fxz} and \eqref{eq:dof-HcurlS2-fyz} that $\sigma_{xy}$ vanishes on the faces $F_{yz}$ and $F_{xz}$. 
 It indicates that $$\sigma_{xy}= z^2(1-z)^2x(1-x)y(1-y)\sigma_2,\quad \text{with}\quad \sigma_{2}\in\cQ_{k-3,k-3,k-4}.$$
 Finally, by \eqref{eq:dof-HcurlS2-t} it indicates that $\sigma_{xy}=0$. 
 
 When $k = 3$, a similar argument shows that $\sigma_{xy} = z^2(1-z)^2 \sigma_2$ for some polynomial $\sigma_{2}$. Since $\sigma_{xy} \in \cQ_{2,2,3}$ it can be concluded that $\sigma_{xy} = 0$. 

The $\bm H(\curl; \bS)$ conformity of $\bm \sigma$ is already implied by the previous proof, and \Cref{lem:HcurlS-conformity}.
\end{proof}
The dimension of $\bm{\Sigma}_{h}$ is 
\begin{equation*}
\begin{aligned}
\dim \bm{\Sigma}_h = &[4(k-1)\mathscr{E} + 4(k-1)(k-3)\mathscr{F} + 3(k-1)(k-3)^2\mathscr{T}]\\
&+[6\mathscr{V} + 4(k-2)\mathscr{E} + (k-3)\mathscr{E} + 2(k-2)(2k-5)\mathscr{F}
% + 2(k-2)(k-3)\mathscr{F}
+ 3(k-2)^2(k-3)\mathscr{T}].
\end{aligned}
\end{equation*}

\subsection{\texorpdfstring{$\bm H(\div;\bT)$}{H(div;T)} conforming finite element space}
\label{sec:HdivT}
This subsection considers the construction of an $\bm H(\div;\bT)$ conforming space $\bm{\Xi}_h$. For this element, the shape function space on $T$ is 
\begin{equation}
\label{eq:shapefunc-HdivT}
\begin{bmatrix}
    \tau_{xx} & \tau_{xy} & \tau_{xz} \\ \tau_{yx} & \tau_{yy} & \tau_{yz} \\ \tau_{zx} & \tau_{zy} & \tau_{zz}
\end{bmatrix}
\in 
\begin{bmatrix}
    \cQ_{k-1,k-1,k-1} & \cQ_{k-2,k,k-1} &\cQ_{k-2,k-1,k} \\
    \cQ_{k,k-2,k-1} & \cQ_{k-1,k-1,k-1} & \cQ_{k-1,k-2,k} \\
    \cQ_{k,k-1,k-2} & \cQ_{k-1,k,k-2} &\cQ_{k-1,k-1,k-1}
\end{bmatrix} =: \Xi_T.
\end{equation}

Like those of the space $\bm \Sigma_T$, the degrees of freedom of $\bm \tau$ are componentwisely defined. Only the construction of the first row $\bm{\tau_x}$ is given, since the degrees of freedom of the remaining components can be defined in a similar way.

\textbf{The degrees of freedom on $\tau_{xx}$} are defined as follows:
\begin{enumerate}
    \item The value of $\tau_{xx}$ at each vertex $\bm{x}$ of $T$, namely, $\tau_{xx}(\bm x)$.
    \item The moments of $\tau_{xx}$ on each edge $e$ of $T$, \begin{equation}\label{eq:dof-HdivT-e}\int_{e} \tau_{xx} p ~dl \for p \in \cQ_{k-3}.\end{equation}
    \item The moments of $\tau_{xx}$ on each face $F$ of $T$, \begin{equation}\label{eq:dof-HdivT-f} \int_{F} \tau_{xx} p ~ds \for p \in \cQ_{k-3,k-3}.\end{equation}
    % for $p \in \cQ_{k-3}(x)$ and 
    % \begin{equation}\int_{e_y} \tau_{xx} p ~dy\end{equation} for $p \in \cQ_{k-3}(y)$
    % and 
    % \begin{equation}\int_{e_z} \tau_{xx} p ~dz\end{equation} for $p \in \cQ_{k-3}(z)$.
    % \item \begin{equation}\int_{F_{xy}} \tau_{xx}p~dxdt\end{equation} for $p \in \cQ_{k-3,k-3}(x,y)$ and 
    % \begin{equation}\int_{F_{xz}} \tau_{xx}p~dxdz\end{equation} for $p \in \cQ_{k-3,k-3}(x,z)$ and 
    %  \begin{equation}\int_{F_{yz}} \tau_{xx}p~dydz\end{equation} for $p \in \cQ_{k-3,k-3}(y,z)$.
     \item The moments of $\tau_{xx}$ inside $T$, \begin{equation}\label{eq:dof-HdivT-t}\int_T \tau_{xx}p~dxdydzfor  \for p \in \cQ_{k-3,k-3,k-3}(x,y,z).\end{equation}
\end{enumerate}
The degrees of freedom of $\tau_{yy}$ and $\tau_{zz}$ can be similarly defined, by a cyclic permutation. 

\begin{remark}
Notice that since the matrix-valued piecewise polynomial $\bm \tau$ is traceless, only two of components $\tau_{xx}, \tau_{yy}$ and $\tau_{zz}$ should be given.
\end{remark}

\textbf{The degrees of freedom of $\tau_{xy}$} are defined as follows:
\begin{enumerate}
    \item The moments of $\tau_{xy}$ and $\py \tau_{xy}$ on each edge $e_x$ of $T$, \begin{equation}\label{eq:dof-HdivT2-ex}\int_{e_x}\tau_{xy}p~dx, ~~~~\int_{e_x}\py \tau_{xy}p ~dx \for p \in \cQ_{k-2}(x).\end{equation}
    \item The moments of $\tau_{xy}$ on each face $F_{xy}$ of $T$, \begin{equation}\label{eq:dof-HdivT2-fxy}\int_{F_{xy}}\tau_{xy}p~dxdy\quad \text{for} ~~p \in \cQ_{k-2,k-4}(x,y).\end{equation} 
    \item The moments of $\tau_{xy}$ and $\py \tau_{xy}$ on each face $F_{xz}$ of $T$, \begin{equation}\label{eq:dof-HdivT2-fxz}\int_{F_{xz}}\tau_{xy}p~dxdz, ~~~~\int_{F_{xz}}\py \tau_{xy}p~dxdy \for p \in \cQ_{k-2,k-3}(x,z).\end{equation}
    \item The moments of $\tau_{xy}$ inside the element $T$, \begin{equation}\label{eq:dof-HdivT2-t}\int_T \tau_{xy}p~dxdy \for p \in \cQ_{k-2,k-4,k-3}(x,y,z).\end{equation}
\end{enumerate}
The degrees of freedom of the remaining five components can be similarly defined via permutation.

The next proposition shows the unisolvency of the degrees of freedom with respect to the shape function spasce, and $\bm H(\div;\mathbb T)$ conformity of the corresponding finite element space.
\begin{proposition}
Suppose $k\ge 3$, the degrees of freedom defined above are unisolvent for the shape function defined in \eqref{eq:shapefunc-HdivT}, and the resulting finite element space $\bm \Xi_h$ is $\bm H(\div; \bT)$ conforming.
\end{proposition}
\begin{proof}

It suffices to prove that $\tau_{xx}$ and $\tau_{xy}$ are unisolvent, since the proof of the remaining components is similar. For the diagonal part, it is the classical Lagrange element. 

\paragraph{Unisolvency of $\tau_{xy}$}
Next, it suffices to prove the unisolvency of $\tau_{xy}$. The number of degrees of freedom of $\tau_{xy}$ is equal to the dimension of the shape function space $\cQ_{k-2,k,k-1}$, that is,
\begin{equation}
        8(k-1) + 2(k-1)(k-3) + 4(k-1)(k-2)+(k-1)(k-3)(k-2) 
        = (k-1)k(k+1) = \operatorname{dim}\cQ_{k-2,k,k-1}.
\end{equation}
For any $ \tau_{xy} \in \cQ_{k-2,k,k-1}$, it can be shown that it vanishes at all degrees of freedom if and only if $\tau_{xy}=0$.

\begin{enumerate}
    \item When $k \ge 4$: From \eqref{eq:dof-HdivT2-ex}, $\tau_{xy}$ and $\frac{\partial \tau_{xy}}{\partial y}$ vanish at all edges $e_{x}$. From \eqref{eq:dof-HdivT2-fxy} and \eqref{eq:dof-HdivT2-fxz}, $\tau_{xy}$ vanishes on face $F_{xy},F_{xz}$,
    % This shows that $\tau_{xy}$ is single-valued on the faces $F_{xz},F_{xy}$,
    Similarly, one can show that $\frac{\partial \tau_{xy}}{\partial y}$ vanishes on $F_{xz}$. It follows that $\tau_{xy} = y^2(1-y)^2z(1-z)\tau_1$ with $\tau_1\in \cQ_{k-2,k-4,k-3}$. Then \eqref{eq:dof-HdivT2-t} implies that $\tau_{xy}=0$. This completes the proof.
    \item  When $k = 3$: From \eqref{eq:dof-HdivT2-ex}, $\tau_{xy}$ and $\py \tau_{xy}$ vanish at all edges $e_x$. Then on $F_{xz}$, it follows that $\tau_{xy} = z(1-z)\tau_1$ for some $\tau_1 \in \cQ_{1,0}(x,z)$. It follows from \eqref{eq:dof-HdivT2-fxz} that $\tau_{xy}$ vanishes on $F_{xz}$. Similarly it can deduced that $\py \tau_{xy}$ vanishes on $F_{xz}$. Hence $\tau_{xy} = y^2(1-y)^2\tau_1$ for some polynomial $\tau_1$ and since $\tau \in \cQ_{2,3,1}(x,y,z)$, yielding that $\tau_{xy} = 0$, which completes the proof.
\end{enumerate}

The $\bm H(\div; \bT)$ conformity comes from the fact that $\tau_{xy}$ is single-valued on the faces $F_{xz}, F_{xy}$ (and the similar results on the remaining off-diagonal entries holds) and the diagonal entries (e.g. $\tau_{xx}$) is continuous.

\end{proof}
The dimension of $\bm{\Xi}_{h}$ is
\begin{equation}
\begin{aligned}
& \dim \bm \Xi_h = [2\mathscr{V} + 2(k-2)\mathscr{E} + 2(k-2)^2\mathscr{F} +2(k-2)^3\mathscr{T}]\\
& + [4(k-1)\mathscr{E} + 2(k-1)(k-3)\mathscr{F} + 4(k-1)(k-2)\mathscr{F} + 6(k-1)(k-2)(k-3)\mathscr{T}].
\end{aligned}
\end{equation}

\subsection{\texorpdfstring{$\bm L^2$}{L2} finite element space}

This subsection considers an $\bm L^2$ conforming vector-valued finite element space $\bm Q_h$. For this element, the shape function on $T$ is
\begin{equation}
    \label{eq:shapefunc-L2V}
\begin{bmatrix} q_x \\ q_y \\ q_z \end{bmatrix} \in
\begin{bmatrix}
    \cQ_{k-2,k-1,k-1} \\ \cQ_{k-1,k-2,k-1} \\\cQ_{k-1,k-1,k-2} \\ 
\end{bmatrix}=: \bm Q_T.
\end{equation}
Given $\bm q \in \bm Q_h$, the degrees of freedom $q_x$ are defined as follows: 

\begin{enumerate}
\item The moments of $q_x$ on each edge $e_x$ of $T$,
\begin{equation}\label{eq:dof-L2V-ex}\int_{e_x}q_xp~dx  \for p \in \cQ_{k-2}(x).\end{equation}
\item The moments of $q_x$ on each face $F_{xy}$ of $T$,
\begin{equation}\label{eq:dof-L2V-fxy}\int_{F_{xy}}q_xp~dxdy \for p \in \cQ_{k-2,k-3}(x,y).\end{equation} 
\item The moments of $q_x$ on each face $F_{xz}$ of $T$, 
\begin{equation}\label{eq:dof-L2V-fxz}\int_{F_{xz}}q_xp~dxdz \for p \in \cQ_{k-2,k-3}(x,z). \end{equation} 
\item The moments of $q_x$ inside the element $T$,
\begin{equation}\label{eq:dof-L2V-t}\int_T q_xp~dx \for p \in \cQ_{k-2,k-3,k-3}(x,y,z).\end{equation}
\end{enumerate}
The degrees of freedom of $q_y$ and $q_z$ can be defined similarly as those of $q_x$ by permutation. 
\begin{proposition}
    Suppose that $k \ge 3$, the degrees of freedom defined above are unisolvent with respect to the shape function space \eqref{eq:shapefunc-L2V}. 
\end{proposition}
\begin{proof}
Here consider the unisolvency of $q_x$ only.
The number of degrees of freedom is equal to the dimension of the shape function space $\cQ_{k-2,k-1,k-1}$, namely,
\begin{equation}
4(k-1) + 2(k-1)(k-2) + 2(k-1)(k-2) + (k-1)(k-2)^2= (k-1)k^2.
\end{equation}
% Hence it suffices to show the unisolvency of $q_x$, since the unisolvency of the remaining two components $q_y$ and $q_z$ is similar. Suppose $q_x$ vanishes for the aforementioned DOFs.
Hence it suffices to prove that for $q_x \in \cQ_{k-2,k-1,k-1}$, vanishing at all the degrees of freedom defined above, then $q_x = 0$.
Since $q_x|_{e_x} \in \cQ_{k-2}$ the degrees of freedom in  \eqref{eq:dof-L2V-ex} indicates that $q_x|_{e_x} = 0$. Restricting $q_x$ on $F_{xy}$, it can be deduced that $q_x|_{F_{xy}} = q_1y(1-y)$ for some $q_1 \in \cQ_{k-2,k-3}(x,y).$ Hence by \eqref{eq:dof-L2V-fxy}, $q_x|_{F_{xy}} = 0$. Similarly by \eqref{eq:dof-L2V-fxz}, $q_x|_{F_{xz}} = 0$. Hence $q_x = q_2y(1-y)z(1-z)$ for some $q_2 \in \cQ_{k-2,k-3,k-3}(x,y,z)$, which leads to $q_x = 0$ by \eqref{eq:dof-L2V-t}.
\end{proof}

The dimension of $\bm{Q}_h$ is 
\begin{equation}
    \dim \bm Q_h = 
(k-1)\mathscr{E} + 2(k-1)(k-2)\mathscr{F} + 3(k-1)(k-2)^2 \mathscr{T}.
\end{equation}

\subsection{Finite element complex and its exactness}
% \label{sec:complex}
This section proves that the following finite element sequence  
\begin{equation}
\label{eq:complex-gradgrad-D}
    \cP_1 \stackrel{\subset}\longrightarrow U_h \stackrel{\grad \grad}{\longrightarrow} \bm{\Sigma}_h \stackrel{\curl}{\longrightarrow} \bm{\Xi}_h \stackrel{\div}{\longrightarrow} \bm{Q}_h \longrightarrow 0
\end{equation}
is an exact complex. Since there hold the following dimensions of these spaces $U_h$, $\bm \Sigma_h$, $\bm \Xi_h$ and $\bm Q_h$ defined in the previous subsections, 
\begin{equation*}
\begin{aligned}
 \dim U_h = &  8\mathscr{V} + 4(k-3)\mathscr{E} + 2(k-3)^2\mathscr{F} + (k-3)^3\mathscr{T},\\
    \dim \bm{\Sigma}_h = &[4(k-1)\mathscr{E} + 4(k-1)(k-3)\mathscr{F} + 3(k-1)(k-3)^2\mathscr{T}],\\
    &+[6\mathscr{V} + 4(k-2)\mathscr{E} + (k-3)\mathscr{E} + 2(k-2)^2\mathscr{F} + 2(k-2)(k-3)\mathscr{F} + 3(k-2)^2(k-3)\mathscr{T}],
\\
       \dim \bm{\Xi}_h =  &[2\mathscr{V} + 2(k-2)\mathscr{E} + 2(k-2)^2\mathscr{F} +2(k-2)^3\mathscr{T}]\\
        & + [4(k-1)\mathscr{E} + 2(k-1)(k-3)\mathscr{F} + 4(k-1)(k-2)\mathscr{F} + 6(k-1)(k-2)(k-3)\mathscr{T}],
\\
        \dim \bm{Q}_h = &(k-1)\mathscr{E} + 2(k-1)(k-2)\mathscr{F} + 3(k-1)(k-2)^2 \mathscr{T},
\end{aligned}
\end{equation*}
this leads to
\begin{equation}
\label{eq:dim-counting-gradgrad}
\begin{split}
\dim \mathcal P_1 - \dim U_h + \dim \bm \Sigma_h - \dim \bm \Xi_h + \dim \bm Q_h =~ & 4 - 4\mathscr{V} + 4\mathscr{E} - 4\mathscr{F} + 4\mathscr{T} \\ =~ & 0
\end{split}
\end{equation}
by Euler's formula.

Before proving the exactness of the sequence in \eqref{eq:complex-gradgrad-D}, two auxiliary spaces (so-called bubble function spaces) on each element $T\in\mathcal{T}_{h}$ are introduced:
\begin{equation*}
\begin{split}
\mathring{\bm \Xi}(T) := \{\bm \tau \in \bm \Xi_{T} \text{ vanishes for the DOFs defined on the boundary }\partial T\},
\end{split}
\end{equation*}
\begin{equation*}
\begin{split}
    \mathring{\bm Q}(T) := \{\bm q \in \bm Q_{T} \text{ vanishes for the DOFs defined on the boundary } \partial T \text{ and } \int_T\bm q = \bm 0\}.
\end{split}
\end{equation*}
Then the following result holds, indicating that the discrete $\div$ operator is surjective for the above two bubble function spaces. The proof is similar to that in the local version, see \Cref{sec:gradgrad:local}.
\begin{lemma}
\label{lem:exactness-gradgrad-surjective-bubble}
	It holds that $\div \mathring{\bm \Xi}(T) = \mathring{\bm Q}(T)$ when $k \ge 3$. 
\end{lemma}
\begin{proof}
		It is easy to see that $\div \mathring{\bm \Xi}(T) \subset \mathring{\bm Q}(T)$. For any $\bm q\in \mathring{\bm Q}(T)$, there exists $\bm \tau_0\in \bm {H}^1(T;\mathbb{T})$ with vanishing trace such that $\div \bm \tau_0=\bm q$. Then define $\bm \tau\in \mathring{\bm \Xi}(T)$ such that
	\begin{equation}
	\int_{T} \tau_{xx}p = \int_{T}(\tau_0)_{xx}p,\quad \forall p\in \cQ_{k-3,k-3,k-3}(T),
	\end{equation}
	and
	\begin{equation}
	\int_{T}\tau_{xy}p = \int_{T}(\tau_0)_{xy}p\quad \forall p\in \cQ_{k-2,k-4,k-3}(T),
	\end{equation}
	while the remaining components can be similarly defined.
    As a result, $\bm \tau \in \mathring{\bm \Xi}(T).$
    Then for any $\bm p=[p_x,p_y,p_z]^{T}$ with $p_{x} \in \cQ_{k-2,k-3,k-3},p_{y} \in \cQ_{k-3,k-2,k-3},p_{z} \in \cQ_{k-3,k-3,k-2}$, it follows that
	\begin{equation*}
	\int_{T}\div \bm \tau\cdot \bm p = \int_{T}\bm \tau:\nabla \bm p =\int_{T}\bm \tau_0:\nabla  \bm p = \int_{T} \div \bm \tau_0\cdot \bm p = \int_{T}\bm q\cdot \bm p.
	\end{equation*}
	It follows from the definition of $\bm \tau$ that $\div \bm \tau$ is of the form 
    $$\div \bm \tau = \begin{bmatrix} y(1-y)z(1-z)\gamma_x\\ x(1-x)z(1-z) \gamma_y \\ x(1-x)y(1-y) \gamma_z\end{bmatrix},$$
    where $\gamma_x \in \cQ_{k-2,k-3,k-3}$ and $\gamma_{y} \in \cQ_{k-3,k-2,k-3},\gamma_{z} \in \cQ_{k-3,k-3,k-2}$ are in the corresponding space.
    Hence $\div \bm \tau = \bm q$.
    
\end{proof}

\begin{proposition}
	\label{prop:exactness-gradgrad-surjective}
	The discrete divergence operator $\div: \bm \Xi_h \to \bm Q_h$ is surjective.
	\end{proposition}
\begin{proof}
The proof is divided into two steps. In the first step, 
for a given $\bm q \in \bm Q_h$, a function $\bm \tau \in \bm \Xi_h$ will be constructed such that $(\div \bm \tau - \bm q)|_T \in \mathring{\bm Q}(T)$ for each $T \in \mathcal T_h$. The construction is as follows:

First, let $\tau_{xx} = \tau_{yy} = \tau_{zz} = 0$. Second, for each $e_{x}$ of $T$, let
% \begin{enumerate}
\begin{equation}\label{eq:sur-div-ex}\int_{e_x} \tau_{xy}p~dx = 0 \for p \in \cQ_{k-2}(x),\end{equation}
and
 \begin{equation}\label{eq:sur-div-ex2}\int_{e_x} \py \tau_{xy}p~dx = \frac{1}{2} \int_{e_x} q_xp~dx \for p \in \cQ_{k-2}(x).\end{equation}

 Third, for two faces $F_{xz}^+$ and $F_{xz}^-$, let 
\begin{equation}\label{eq:sur-div-fxz}\int_{F_{xz}^+} \tau_{xy} ~dxdz - \int_{F_{xz}^-}\tau_{xy}~dxdz = \frac{1}{2} \int_{T} q_x.\end{equation}

Fourth, for each face $F_{xz}$, which is normal to the $y-$direction, let
\begin{equation}\label{eq:sur-div-fxz2}\int_{F_{xz}} \py \tau_{xy} p ~dxdy = \int_{F_{xz}} q_xp ~dxdy \for p \in \cQ_{k-2,k-3}(x,z).\end{equation}

Last, for each face $F_{xy}$, set
\begin{equation}\label{eq:sur-div-fxy} \int_{F_{xy}} \tau_{xy} p ~dxdy = 0 \for p \in \cQ_{k-2,k-4}(x,y).\end{equation}
% \end{enumerate}
% The definition of $\tau_{xz}$ are similar, but change the role of $y$ and $z$. Choose $\tau_{xx} = 0$ to keep the zero trace, then set $\bm \tau_x = [\tau_{xx}, \tau_{xy}, \tau_{xz}]$. 
The remaining degrees of freedom of $\tau_{xy}$ can be treated as zero. The component of $\tau_{xz}$ can be similiarly defined.
Next, it can be verified that $(\div \bm \tau_x - q_x)|_{T} \in \mathring{\bm Q}_T$.

It follows from \eqref{eq:sur-div-ex} and \eqref{eq:sur-div-ex2} that
\begin{equation}\int_{e_x} ( \py \tau_{xy} + \pz \tau_{xz} - q)p ~dxdz = 0 \for p \in \cQ_{k-2}(x).\end{equation} This shows that the degrees of freedom \eqref{eq:dof-L2V-ex} vanish for $\div \bm \tau - \bm q.$

By \eqref{eq:sur-div-ex}, \eqref{eq:sur-div-fxz2}, \eqref{eq:sur-div-fxy}, an integration by parts yields, for any $p\in Q_{k-2,k-3}(x,z)$ that
\begin{equation}
\begin{split}
\int_{F_{xz}}( \py \tau_{xy} \!+ \pz \tau_{xz} \!- q_x)p ~dxdz = & \int_{F_{xz}} (\py \tau_{xy} \!- q_x )p~dxdz \\ & +  \int_{e_{x_1}} \tau_{xz} p dx - \int_{e_{x_0}} \tau_{xz} p dx - \int_{F_{xz}}\tau_{xz}\pz p~dxdz \\
= & 0.
\end{split}
\end{equation}
It indicates the degrees of freedom in \eqref{eq:dof-L2V-fxy} and \eqref{eq:dof-L2V-fxz} vanish for $\div \bm \tau - \bm q.$

Now it suffices to show the construction is of mean zero inside the element $T$, which is directly from the following calculation
\begin{equation}
\begin{split}
    & \int_{T}( \py \tau_{xy} + \pz \tau_{xz} - q_x)~dxdydz \\ & 
    (\int_{F_{xz}^+}\tau_{xy}~dxdz -  \int_{F_{xz}^- }\tau_{xy}~dxdz)+ (\int_{F_{xy}^+}\tau_{xz}~dxdy- \int_{F_{xy}^-}\tau_{xz}~dxdy) -\int_{T}q_x~dxdydz \\ 
    = & 0.
\end{split}
\end{equation}

By Lemma~\ref{lem:exactness-gradgrad-surjective-bubble}, there exists $\tilde {\bm \tau}\in \bm\Xi_h$, such that $\div \tilde{\bm \tau} = \bm q - \div \bm \tau$.
Therefore, $\bm \tau + \tilde{\bm \tau} \in \bm \Xi_h$ and $\div(\bm \tau+\tilde{\bm \tau}) = \bm q$. Hence the proof of the proposition is complete.
\end{proof}

To show $U_h \stackrel{\grad \grad}{\longrightarrow} \bm \Sigma_h \stackrel{\curl}{\longrightarrow} \bm \Xi_h $ is exact, it suffices to prove that $\ker \curl = \im \grad \grad$ on the discrete level for those finite element spaces.
\begin{proposition}
\label{prop:exactness-gradgrad-first}
    Suppose $\bm \sigma \in \bm \Sigma_h$ such that $\curl \bm \sigma = 0$, then $\bm \sigma = \grad \grad u$ for some $u \in U_h$. That is, $\ker \curl = \im \grad \grad$ on the discrete level.
\end{proposition}
\begin{proof}
Since $\curl \bm \sigma = 0$, there exists $u \in H^2(\Omega)$ such that $\grad \grad u = \bm \sigma$, from \eqref{eq:complex-gradgrad-C}. Since $\bm \sigma|_T \in \bm \Sigma_h$ is a symmetric matrix-valued polynomial for each element $T \in \mathcal T$, $u|_T$ is a polynomial of degree $k$ for all variables $x,y,z$. 
Combining with $u \in H^2$, it follows that $u \in C^1(\Omega)$.
To show that $u \in U_h$, it suffices to check the higher order continuity across the internal vertices, edges, and faces of $\mathcal T_h$. It follows from $$\pxyz u = \pz \sigma_{xy}, \pxy u = \sigma_{xy}, \pyz u = \sigma_{yz} \text{ and }\pxz u = \sigma_{xz},$$ that the degrees of freedom of $u$ defined at vertices are single-valued by the definition of $\bm \Sigma_h$. A similar argument shows the continuity of the following moments $$\int_{e_x} \pyz u p ~dx = \int_{e_x} \sigma_{yz} p dx$$ for any $p \in \cQ_{k-4}(x)$. Hence it holds that $u \in U_h$.

\end{proof}

A combination of  the dimension counting \eqref{eq:dim-counting-gradgrad}, Proposition~\ref{prop:exactness-gradgrad-surjective} and Proposition~\ref{prop:exactness-gradgrad-first} indicates the following result. 
\begin{theorem}
The finite element sequence in \eqref{eq:complex-gradgrad-D} is an exact complex.
\end{theorem}

\subsection{A discrete complex with reduced regularity}
In this section, a new $\bm H(\curl; \bS)$ conforming space and a new $\bm H(\div; \bT)$ conforming space with reduced regularity will be constructed.
\subsubsection{\texorpdfstring{$\bm H(\curl;\bS)$}{H(curl;S)} conforming finite element space with reduced regularity}

    This subsection considers an $\bm H(\curl;\bS)$ conforming finite element space $\bm{\widetilde{\Sigma}}_h$ with reduced regularity. For this element, the shape function space on $T$ is
    \begin{equation}
    \label{eq:shapefunc-HcurlS-new}
    \begin{bmatrix}
    \sigma_{xx} & \sigma_{xy} & \sigma_{xz} \\
    \sigma_{yx} & \sigma_{yy} & \sigma_{yz} \\
    \sigma_{zx} & \sigma_{zy} & \sigma_{zz}
    \end{bmatrix}
    \in
    \begin{bmatrix}
        \cQ_{k-2,k,k} & \cQ_{k-1,k-1,k} & \cQ_{k-1,k,k-1}\\
        \cQ_{k-1,k-1,k} & \cQ_{k,k-2,k} & \cQ_{k,k-1,k-1} \\
        \cQ_{k-1,k,k-1} & \cQ_{k,k-1,k-1} & \cQ_{k,k,k-2} 
    \end{bmatrix} =: \bm \Sigma_T.
    \end{equation}

    For $\bm \sigma \in \bm{\widetilde \Sigma}_h$, due to its symmetry, the degrees of freedom will be separated into six parts, $$\sigma_{xx}, \sigma_{yy}, \sigma_{zz}, \sigma_{xy} = \sigma_{yx}, \sigma_{xz} = \sigma_{zx}, \sigma_{yz} = \sigma_{zy}.$$ In the follows, only the degrees of freedom of $\sigma_{xx}$ and $\sigma_{xy}$ will be specified, since those of the remaining four components are similarly defined via cyclic permutation of the index.
    
    \textbf{The degrees of freedom of $\sigma_{xx}$} are defined as follows:
    \begin{enumerate}
        \item The moments of $\sigma_{xx}$ on each edge $e_x$ of $T$, \begin{equation}\label{eq:dof-HcurlS-new-e}
        \int_{e_x} \sigma_{xx}p~dx, 
        %  {\color{orange} {\int_{e_x} \py \sigma_{xx}p~dx, \int_{e_x} \pz \sigma_{xx} p~dx}}
          \text{ for }p \in \cQ_{k-2}(x).
        \end{equation}
        \item The moments of $\sigma_{xx}$ on each face $F_{xy}$ of $T$,
        \begin{equation}\label{eq:dof-HcurlS-new-fxy}
        \int_{F_{xy}} \sigma_{xx}p~dxdy,\for p \in \cQ_{k-2,k-2}(x,y).
        \end{equation}
        \item The moments of $\sigma_{xx}$ on each face $F_{xz}$ of $T$,
        \begin{equation}\label{eq:dof-HcurlS-new-fxz}
             \int_{F_{xz}} \sigma_{xx}p~dxdz, \for p \in \cQ_{k-2,k-2}(x,z).\end{equation}
        \item The moments of $\sigma_{xx}$ inside the element $T$, \begin{equation}\label{eq:dof-HcurlS-new-t} \int_{T} \sigma_{xx}p ~dxdydz \for p \in \cQ_{k-2,k-2,k-2}(x,y,z).\end{equation}
    
    \end{enumerate}
     % The degrees of freedom of $\sigma_{yy}, \sigma_{zz}$ are similarly defined by cyclic permutation of the index. 

Compared to the version provided in \Cref{sec:HcurlS}, the reduced element relaxes some partial regularity on edges.

    \textbf{The degrees of freedom of $\sigma_{xy}$} are defined as follows:
    \begin{enumerate}
        \item The value and partial derivative along $z$-direction of $\sigma_{xy}$ at each vertex $\bm x$ of $T$, \begin{equation}\label{eq:dof-HcurlS2-new-v}\sigma_{xy}(\bm{x}),  \pz\sigma_{xy}(\bm{x}).\end{equation}
        \item The moments of $\sigma_{xy}$ on each edge $e_x$ of $T$,
        \begin{equation}\label{eq:dof-HcurlS2-new-ex}
        \int_{e_x} \sigma_{xy} p ~dx,  {\int_{e_x} \pz \sigma_{xy}}\for p \in \cQ_{k-3}(x).
        \end{equation}
        \item The moments of $\sigma_{xy}$ on each edge $e_y$ of $T$, 
         \begin{equation}\label{eq:dof-HcurlS2-new-ey}\int_{e_y} \sigma_{xy} p ~dy, {\int_{e_y} \pz \sigma_{xy}} \for  p \in \cQ_{k-3}(y).\end{equation}
         \item The moments of $\sigma_{xy}$ on each edge $e_z$ of $T$, \begin{equation}\label{eq:dof-HcurlS2-new-ez}\int_{e_z} \sigma_{xy} p~dx \for  p \in \cQ_{k-4}(z).\end{equation}
         
         \item The moments of $\sigma_{xy}$ on each face $F_{xy}$ of $T$,\begin{equation}\label{eq:dof-HcurlS2-new-fxy}\int_{F_{xy}} \sigma_{xy} p ~dxdy,  \for p \in \cQ_{k-3,k-3}(x,y).\end{equation}
         \item The moments of $\sigma_{xy}$ on each face $F_{xz}$ of $T$,  \begin{equation}\label{eq:dof-HcurlS2-new-fxz}\int_{F_{xz}} \sigma_{xy} p ~dxdz \for p \in \cQ_{k-3,k-4}(x,z).\end{equation} 
         \item The moments of $\sigma_{xy}$ on each face $F_{yz}$ of $T$, 
         \begin{equation}\label{eq:dof-HcurlS2-new-fyz}\int_{F_{yz}} \sigma_{xy} p~dydz \for p \in \cQ_{k-3,k-4}(y,z).\end{equation}
         \item The moments of $\sigma_{xy}$ inside the element $T$, \begin{equation}\label{eq:dof-HcurlS2-new-t}\int_T \sigma_{xy} p ~dxdydz \for p \in \cQ_{k-3,k-3,k-2}(x,y,z).\end{equation}
    \end{enumerate}
    % When $k = 3$, \eqref{eq:dof-HcurlS2-new-ez}, \eqref{eq:dof-HcurlS2-new-fxz},\eqref{eq:dof-HcurlS2-new-fyz} and \eqref{eq:dof-HcurlS2-new-t} disappear. 
    % The degrees of freedom of $\sigma_{yz}, \sigma_{zx}$ can be similarly defined as $\sigma_{xy}$ via premutation. 

 Compared to the finite element in \Cref{sec:HdivT}, the major difference here is in the fifth set of the degrees of freedom.

   The next proposition shows the unisolvency of the degrees of freedom with respect to the shape function spaces \eqref{eq:shapefunc-HcurlS-new}, and the corresponding finite element space is $H(\curl; \bS)$ conforming.

    \begin{proposition}
    The above degrees of freedom are unisolvent with respect to the shape function space $\Sigma_{T}$ in \eqref{eq:shapefunc-HcurlS-new}, and the resulting finite element space $\bm{\widetilde \Sigma}_h$ is $\bm H(\curl; \bS)$ conforming.
    \end{proposition}
    \begin{proof}
        The unisolvency of $\sigma_{xx}$ is classical. For the component $\sigma_{xy}$,  the proof is similar as that of $\sigma_{xy}$ of $\bm \Sigma_h$. The $\bm H(\curl, \bS)$ conformity is implied in the previous proof.
    \end{proof}
    The dimension of $\bm{\widetilde \Sigma}_{h}$ is 
    \begin{equation*}
    \begin{aligned}
    &\dim \bm{\widetilde \Sigma}_{h}= [(k-1)\mathscr{E} + 2(k-1)^2\mathscr{F} + 3(k-1)^3\mathscr{T}]\\
    &+[6\mathscr{V} + 4(k-2)\mathscr{E} + (k-3)\mathscr{E} + (k-2)^2\mathscr{F} + 2(k-2)(k-3)\mathscr{F} + 3(k-2)^2(k-1)\mathscr{T}].
    \end{aligned}
    \end{equation*}

    \subsubsection{\texorpdfstring{$\bm H(\div;\bT)$}{H(div;T)} conforming finite element space with reduced regularity}

    This subsection considers a new $\bm H(\div; \bT)$ conforming finite element space $\bm{\widetilde{\Xi}}_h$ with reduced regularity. The shape function space on $T$ is as follows:
\begin{equation}
    \label{eq:shapefunc-HdivT-new}
    \begin{bmatrix}
        \tau_{xx} & \tau_{xy} & \tau_{xz} \\ \tau_{yx} & \tau_{xy} & \tau_{yz} \\ \tau_{zx} & \tau_{zy} & \tau_{zz}
    \end{bmatrix}
    \in 
    \begin{bmatrix}
        \cQ_{k-1,k-1,k-1} & \cQ_{k-2,k,k-1} &\cQ_{k-2,k-1,k} \\
        \cQ_{k,k-2,k-1} & \cQ_{k-1,k-1,k-1} & \cQ_{k-1,k-2,k} \\
        \cQ_{k,k-1,k-2} & \cQ_{k-1,k,k-2} &\cQ_{k-1,k-1,k-1}
    \end{bmatrix} = \bm \Xi_T.
    \end{equation}

The following bubble space on $T$ is needed for the following construction:
\begin{equation}
\begin{split}
\mathcal{B}_{\div,k;\bT}(T) =\big\{ &(\tau_{xx},\tau_{yy}, \tau_{zz}) \in \cQ_{k-1,k-1,k-1} \times \cQ_{k-1,k-1,k-1} \times \cQ_{k-1,k-1,k-1} : \\ &  \tau_{xx} \text{ vanishes on } F_{yz}, \tau_{yy} \text{ vanishes on } F_{xz},  \tau_{zz} \text{ vanishes on } F_{xy}, \\ &  \text{ and } \tau_{xx} + \tau_{yy} + \tau_{zz} = 0\big\}.  
\end{split}
\end{equation}
% Then we have the following lemma.
\begin{lemma}[Dimension of $\mathcal{B}_{\div,k;\bT}(T)$] It holds that
\begin{equation}
\dim \mathcal{B}_{\div, k ; \bT}(T) = 2(k-2)^2(k+1).
\end{equation}
\end{lemma}
\begin{proof}
   Since $x(1-x)|\tau_{xx}$, $y(1-y) | \tau_{yy}$, $z(1-z) | \tau_{zz}$, and $\tau_{xx} + \tau_{yy} +\tau_{zz} = 0$, there exist $p \in \cQ_{k-3,k-1,k-1}$ and $q \in \cQ_{k-1,k-3,k-1}$ such that 
% The pair of bubble is exactly the pair of $p \in \cQ_{k-3,k-1,k-1}$ and $q \in \cQ_{k-1,k-3,k-1}$ satisfying 
\begin{equation}
x(1-x)p(x,y,0) + y(1-y)q(x,y,0) = 0
\end{equation}
and 
\begin{equation}
    x(1-x)p(x,y,1) + y(1-y)q(x,y,1) = 0.
\end{equation}
It follows that there exist $f(x,y) \in \cQ_{k-3,k-3}(x,y)$ and $g(x,y) \in \cQ_{k-3,k-3}(x,y)$ such that 
$$p(x,y,0) = y(1-y)f(x,y), \qquad q(x,y,0) = -x(1-x)f(x,y)$$
and
$$p(x,y,1) = y(1-y)g(x,y), \qquad q(x,y,1) = -x(1-x)g(x,y)$$
for $f,g \in \cQ_{k-3,k-3}(x,y)$. 

% Then, $p$ can be can reconstructed from polynomial-valued Lagrange interpolation.
 Regard $p$ as a polynomial of $z$, of degree $\le k-1$ and coefficients in $\cQ_{k-3,k-1}(x,y)$. Since the values of $p$ at $z=0$ and $z=1$ are given as above, additional values of $p(z_i)$ at $k-2$ points $z_i,i = 1,2,\cdots,k-2,$ uniquely determine $p$.
Therefore, the dimension of $\mathcal{B}_{\div,k;\bT}(T)$ is 
\begin{equation}
2(k-2)^2k + 2(k-2)^2 = 2(k-2)^2(k+1).
\end{equation}
\end{proof}
Based on the above observations, now it is ready to define the degrees of freedom.

\textbf{The degrees of freedom of $(\tau_{xx}, \tau_{yy}, \tau_{zz})$} are defined as follows:

\begin{enumerate}
    \item The values $\tau_{xx}(\bm{x}), \tau_{yy}(\bm{x}), \tau_{zz}(\bm{x})$ at each vertex $\bm{x}$ of element $ T$.
    \item The moments of $\tau_{xx}, \tau_{yy}, \tau_{zz}$ on each edge $e$ of element $T$, 
    \begin{equation}\label{eq:dof-HdivT-new-e}\int_{e} \tau_{xx} p ~dl, \int_{e} \tau_{yy} p ~dl,  \int_{e} \tau_{zz} p ~dl  \for p \in \cQ_{k-3}.\end{equation}

    \item The moments of $\tau_{xx}$ on each face $F_{yz}$ of element $T$, \begin{equation}\label{eq:dof-HdivT-new-f} \int_{F_{yz}} \tau_{xx} p ~ds \for p \in \cQ_{k-3,k-3}(y,z).\end{equation}
    A similar set of degrees of freedom can be defined for $\tau_{yy}$ and $\tau_{zz}$ by cyclic permuatation of the index.
    \item The following moments inside the element $T$, 
\begin{equation}
    \label{eq:dof-HdivT-new-t}
\int_T (\tau_{xx}\xi_{xx} + \tau_{yy}\xi_{yy} + \tau_{zz}\xi_{zz}) dxdydz \for (\xi_{xx}, \xi_{yy}, \xi_{zz}) \in \mathcal{B}_{\div,k;\bT}(T).
\end{equation}

\end{enumerate}

% and the shape function space of $(\tau_{xx}, \tau_{yy}, \tau_{zz})$ is taken as 
% for
% \begin{equation}
%     \label{eq:HcurlS-new-shape-function}
%     (\tau_{xx}, \tau_{yy}, \tau_{zz}) := \{ \tau_{xx}, \tau_{yy}, \tau_{zz} \in \cQ_{k-1,k-1,k-1}~:~ \tau_{xx} + \tau_{yy} + \tau_{zz} = 0\}.
% \end{equation}

\textbf{The degrees of freedom on $\tau_{xy}$} are defined as follows:
\begin{enumerate}
    % \item The moments of $\tau_{xy}$ on each edge $e_x$ of $T$, parallel to $x$-direction, \begin{equation}\label{eq:dof-HdivT2-new-ex}\int_{e_x}\tau_{xy}p~dx, \for p \in \cQ_{k-2}(x).\end{equation}
    % \item For each face $F_{xy}$ of $T$, normal to $z$-direction, \begin{equation}\label{eq:dof-hdiv2-fxy}\int_{F_{xy}}\tau_{xy}p~dxdy\end{equation} for $p \in \cQ_{k-2,k-2}(x,y)$.
    \item The moments of $\tau_{xy}$ each face $F_{xz}$ of $T$, \begin{equation}\label{eq:dof-HdivT2-new-fxz}\int_{F_{xz}}\tau_{xy}p~dxdz,  \for p \in \cQ_{k-2,k-1}(x,z).\end{equation}
    \item The moments of $\tau_{xy}$ inside the element $T$, \begin{equation}\label{eq:dof-HdivT2-new-t}\int_T \tau_{xy}p~dxdy \for p \in \cQ_{k-2,k-2,k-1}(x,y,z).\end{equation}
\end{enumerate}

The degrees of freedom for the remaining off-diagonal components can be similarly defined. 

The unisolvency and desired conformity of the finite element is stated in the following proposition.
\begin{proposition}
% $\bm{\widetilde{\Xi}}_h$ is unisolvent, and $\bm H(\div; \bT)$ conforming.
The above degrees of freedom are unisolvent for the shape function $\Xi_T$, defined in \eqref{eq:shapefunc-HdivT-new}. Moreover, the resulting finite element space $\bm{\widetilde{\Xi}}_h$ is $\bm H(\div; \bT)$ conforming.
\end{proposition}
\begin{proof}
\textbf{Unisolvency of $(\tau_{xx}, \tau_{yy}, \tau_{zz})$}. For this case, the degrees of freedom of these three diagonal components are coupled. Note that the dimension of the shape function space of $(\tau_{xx}, \tau_{yy}, \tau_{zz})$ is $2k^3$. While the total number of degrees of freedom is 
$$ 16 + 24(k-2) + 6(k-2)^2 + 2(k-2)^2(k+1) = 2k^3,$$
equals to the dimension of $(\tau_{xx}, \tau_{yy}, \tau_{zz})$.
Hence, it suffices to show that if $(\tau_{xx}, \tau_{yy}, \tau_{zz})$ belonging to  $$ \{ (\tau_{xx},\tau_{yy},\tau_{zz}) \in \cQ_{k-1,k-1,k-1} \times \cQ_{k-1,k-1,k-1} \times \cQ_{k-1,k-1,k-1} ~:~ \tau_{xx} + \tau_{yy} + \tau_{zz} =  0 \}$$ vanishes for all the DOFs, then $\tau_{xx} = \tau_{yy} = \tau_{zz} = 0$.  By \eqref{eq:dof-HdivT-new-e}, \eqref{eq:dof-HdivT-new-f}, it holds that $\tau_{xx}$ vanishes on face $F_{yz}$, $\tau_{yy}$ vanishes on face $F_{xz}$ and $\tau_{zz}$ vanishes on face $F_{xy}$. Therefore, this yields $(\tau_{xx}, \tau_{yy}, \tau_{zz}) \in \mathcal{B}_{\div, k ; \bT}(T)$. At end, the fourth set of DOFs in \eqref{eq:dof-HdivT-new-t} indicates that $\tau_{xx} = \tau_{yy} = \tau_{zz} = 0$.

% \textbf{Unisolvency of $\tau_{xy}$} It follows from the first, second sets of degrees of freedom that $\tau_{xy} = y(1-y) u$ for $u \in \cQ_{k-2,k-2.k-1}$. Then the last set of degrees of freedom completes the proof.
The unisolvency of $\tau_{xy}$ is straightforward and hence omitted.

Next, consider the $\bm H(\div; \bT)$ conformity of the resulting spacee. This is from (1) $\tau_{xx}$ is single-valued on $F_{yz}$, and (2) $\tau_{xy}$ is single-valued on $F_{xz}$, and the other conditions obtained via a cyclic permutation.
\end{proof}

The dimension of $\bm{\widetilde{\Xi}}_h$ is 

\begin{equation}
    \begin{aligned}
    &\dim \bm{\widetilde{\Xi}}_h =  [2\mathscr{V} + 2(k-2)\mathscr{E} + (k-2)^2\mathscr{F} +2(k-2)^2(k+1)\mathscr{T}]\\
    & + [ 2(k-1)k\mathscr{F} + 6(k-1)^2k\mathscr{T}].
    \end{aligned}
    \end{equation}

    \subsubsection{A new finite element complex and its exactness}

    For this case, the finite element subspace $\bm{\widetilde{Q}}_h$ of $\bm L^2(\Omega, \mathbb R^3)$ is the space of discontinuous piecewise polynomials.  For this element, the shape function space on $T$ is 

\begin{equation}
\begin{bmatrix} q_x \\ q_y \\ q_z \end{bmatrix} \in
\begin{bmatrix}
    \mathcal{DG}_{k-2,k-1,k-1} \\ \mathcal{DG}_{k-1,k-2,k-1} \\\mathcal{DG}_{k-1,k-1,k-2} \\ 
\end{bmatrix},
\end{equation}

This subsubsection aims at proving the following finite element sequence 
\begin{equation}
\label{eq:complex-gradgrad-D-new}
    \cP_1 \stackrel{\subset}\longrightarrow U_h \stackrel{\grad \grad}{\longrightarrow} \bm{\widetilde{\Sigma}}_h \stackrel{\curl}{\longrightarrow} \bm{\widetilde{\Xi}}_h \stackrel{\div}{\longrightarrow} \bm{\widetilde{Q}}_h \longrightarrow 0
\end{equation}
is exact.

Since the dimensions of the finite element space $U_h$, $\bm{\widetilde{\Sigma}}_h$, $\bm{\widetilde{\Xi}}_h$ and $\bm{\widetilde{Q}}_h$ are as follows:
\begin{equation*}
    \begin{aligned}
     \dim U_h = &  8\mathscr{V} + 4(k-3)\mathscr{E} + 2(k-3)^2\mathscr{F} + (k-3)^3\mathscr{T},\\
     \dim \bm{\widetilde \Sigma}_{h}= & [(k-1)\mathscr{E} + 2(k-1)^2\mathscr{F} + 3(k-1)^3\mathscr{T}]\\
     &+[6\mathscr{V} + 4(k-2)\mathscr{E} + (k-3)\mathscr{E} + (k-2)^2\mathscr{F} + 2(k-2)(k-3)\mathscr{F} + 3(k-2)^2(k-1)\mathscr{T}].
        % &+[6\mathscr{V} + 4(k-2)\mathscr{E} + (k-3)\mathscr{E} + (k-2)^2\mathscr{F} + 2(k-2)(k-3)\mathscr{F} + 3(k-2)^2(k-1)\mathscr{T}],
    \\
           \dim \bm{\widetilde{\Xi}}_h =  &[2\mathscr{V} + 2(k-2)\mathscr{E} + (k-2)^2\mathscr{F} +2(k-2)^2(k+1)\mathscr{T}]\\
            &  + [ 2(k-1)k\mathscr{F} + 6(k-1)^2k\mathscr{T}],
    \\
            \dim \bm{\widetilde{Q}}_h = &3(k-1)k^2 \mathscr{T},
    \end{aligned}
    \end{equation*}
this leads to
\begin{equation}
    \label{eq:dim-counting-gradgrad-new}
    \begin{split}
\dim U_h - \dim \bm {\widetilde{\Sigma}}_h + \dim \bm{\widetilde{\Xi}}_h - \dim \bm{\widetilde{Q}}_h = & 4(\mathscr{V} - \mathscr E + \mathscr F - \mathscr T) \\ = & 4 =  \dim  \bm{\mathcal{P}}_1,
    \end{split}
\end{equation}
by Euler's formula.

\begin{proposition}
    \label{prop:exactness-gradgrad-new-surjective}
	Suppose that $k\ge 3$, the discrete divergence operator $\div : \widetilde{\bm \Xi}_h \to \widetilde{\bm Q}_h$ is surjective.
\end{proposition} 

\begin{proof}
    Given $\bm q = (q_x, q_y, q_z)^T \in \bm{\widetilde{Q}}_h$, 
set $\tau_{xy}$ by 
$
\tau_{xy}(x,y,z) = \int_{0}^y \bm q_x ds,
$
and similarly define $\tau_{zx}$ and $\tau_{yz}$. The remaining entries are set zero. Clearly, $\bm \tau$ is $\bm H(\div;\bT)$ conforming, and belongs to the space $\bm{\widetilde{\Xi}}_h$. In addition, $\div \bm \tau = \bm q$. 
\end{proof}

\begin{proposition}
    \label{prop:exactness-gradgrad-new-first}
For $k \ge 3$, if $\bm \sigma \in \bm{\widetilde{\Sigma}}_h$ satisfies that $\curl \bm \sigma = 0$, then there exists $u \in U_h$ such that $\grad \grad u = \bm \sigma$.
\end{proposition}

\begin{proof}
The proof is similar to that in \Cref{prop:exactness-gradgrad-first}. Let a piecewise polynomial function $u \in H^2$ such that $\grad \grad u = \bm \sigma$.
Again, it suffices to check that $\pyz{u}$ is single-valued on edge $e_x$. This directly comes from the fact $\sigma_{yz}$ is single-valued on $e_x$.

\end{proof}

\begin{theorem}
The finite element sequence in \eqref{eq:complex-gradgrad-D-new} is an exact complex, provided $k \ge 3$.
\end{theorem}
\begin{proof}
By dimension counting in \eqref{eq:dim-counting-gradgrad-new}, the result follows from \Cref{prop:exactness-gradgrad-new-surjective} and \Cref{prop:exactness-gradgrad-new-first}.
\end{proof}

\section{Discrete elasticity complex}

This section considers four type of finite element spaces: $\bm H^1(\Omega)$ conforming space $V_h$, $\bm H(\operatorname{curl} \operatorname{curl}^{\mathsf{T}}; \mathbb{S})$ conforming space $\bm{\Phi}_h$,  $\bm H(\operatorname{div}; \mathbb{S})$ conforming space $\bm \Gamma_h $ and $\bm L^2(\Omega)$ space $\bm Z_h$. These finite element spaces will be used to construct

\begin{equation}
    %\label{eq:divdivcomplex-dis}
    \bm{\mathcal{RM}} \stackrel{\subseteq}{\longrightarrow} \bm X_h \stackrel{\operatorname{sym} \operatorname{grad}}{\longrightarrow} \bm{\Phi}_h \stackrel{ \operatorname{curl} \operatorname{curl}^{\mathsf{T}} }{\longrightarrow} \bm \Gamma_h \stackrel{\div}{\longrightarrow} \bm Z_h\longrightarrow 0,
    \end{equation}
    a discrete subcomplex on the cuboid grids of the following continuous elasticity complex,
    \begin{equation}\label{eq:complex-elasticity-C}
        \bm{\mathcal{RM}} \stackrel{\subseteq}{\longrightarrow} \bm H^{1}\left(\Omega ; \bR^3\right) \stackrel{\operatorname{sym} \operatorname{grad}}{\longrightarrow} \bm H(\operatorname{curl} \operatorname{curl}^{\mathsf{T}}, \Omega ; \mathbb{S}) \stackrel{\operatorname{curl} \operatorname{curl}^{\mathsf{T}}}{\longrightarrow} \bm H(\operatorname{div}, \Omega ; \mathbb{S}) \stackrel{\div}{\longrightarrow} \bm L^{2}(\Omega; \mathbb{R}^3) \longrightarrow 0.
        \end{equation}
    Here
    $
    \bm{\mathcal{RM}} = \{ \bm a + \bm b \times \bm x : \bm a, \bm b \in \mathbb{R}^3\}
    $ is the rigid motion space (dim = 6), the spaces 
    $$\bm H(\operatorname{curl} \operatorname{curl}^{\mathsf{T}}, \Omega ; \mathbb{S}) := \{ \bm u \in \bm L^2(\Omega; \bS); \curl \curl^{\mathsf T} \bm u \in \bm L^2(\Omega; \bS) \} $$
    and 
    $$\bm H(\div,\Omega; \bS) := \{ \bm \sigma \in \bm L^2(\Omega; \bS) : \div \bm \sigma \in \bm L^2(\Omega)\}.$$

The following vector identity will be used in the following proof. 
\begin{lemma}

It holds that 
\begin{equation} \label{eq:curlsymgrad} \curl \sym \grad \bm u = \frac{1}{2}(\grad \curl \bm u)^{\mathsf T},
\end{equation}
and therefore 
\begin{equation} \label{eq:curlTsymgrad} \curl ^{\mathsf T} \sym \grad \bm u = \frac{1}{2}\grad \curl \bm u.
\end{equation}
\end{lemma}
\begin{proof}
A straightforward calculation yields that 
$$ \curl^{\mathsf T} \grad \bm u = \grad \curl \bm u.$$ Then,
\begin{equation*}
\begin{split}
\curl \sym \grad \bm u = & \frac{1}{2} \curl \grad \bm u + \frac{1}{2}  \curl (\grad \bm u)^{\mathsf T}  \\ & = 
\frac{1}{2}  (\curl^{\mathsf T} \grad \bm u)^{\mathsf T} \\&  = \frac{1}{2} (\grad \curl \bm u)^{\mathsf T}.
\end{split}
\end{equation*}
\end{proof}

\subsection{Local version: the polynomial complex}
In this subsection, the local version of the finite element complex is constructed with the following form, (where $T = [0,1]^3$)
\begin{equation}
    \label{eq:complex-elasticity-poly}
    \bm{\mathcal{RM}} \stackrel{\subseteq}{\longrightarrow} \bm X_T \stackrel{\operatorname{sym} \operatorname{grad}}{\longrightarrow} \bm{\Phi}_T \stackrel{ \operatorname{curl} \operatorname{curl}^{\mathsf{T}} }{\longrightarrow} \bm \Gamma_T \stackrel{\div}{\longrightarrow} \bm Z_T\longrightarrow 0,
    \end{equation}
    where $k \ge 2$, and
    $$\bm X_T = \begin{bmatrix}
        \cQ_{k,k+1,k+1} \\ \cQ_{k+1,k,k+1} \\\cQ_{k+1,k+1,k} \\ 
    \end{bmatrix} ,$$
    $$\bm \Phi_T = \begin{bmatrix}
        \cQ_{k-1,k+1,k+1} & \cQ_{k,k,k+1} & \cQ_{k,k+1,k}\\
        \cQ_{k,k,k+1} & \cQ_{k+1,k-1,k+1} & \cQ_{k+1,k,k} \\
        \cQ_{k,k+1,k} & \cQ_{k+1,k,k} & \cQ_{k+1,k+1,k-1}
    \end{bmatrix} ,$$
    $$\bm \Gamma_T =     \begin{bmatrix}
        \cQ_{k+1,k-1,k-1} & \cQ_{k,k,k-1} & \cQ_{k,k-1,k}\\
        \cQ_{k,k,k-1} & \cQ_{k-1,k+1,k-1} & \cQ_{k-1,k,k} \\
        \cQ_{k,k-1,k} & \cQ_{k-1,k,k} & \cQ_{k-1,k-1,k+1}
    \end{bmatrix},$$
    and 
    $$\bm Z_T =     \begin{bmatrix} \cQ_{k,k-1,k-1} \\ \cQ_{k-1,k,k-1} \\\cQ_{k-1,k-1,k} \end{bmatrix}. $$

Clearly, the above sequence is a complex. The following proposition indicates that the complex is exact.
\begin{proposition}
The polynomial sequence \eqref{eq:complex-elasticity-poly} is an exact complex. 
\end{proposition}
\begin{proof}
By the choice of the shape function spaces, the polynomial sequence is a complex. It remains to prove its exactness.
First, for $\bm q \in \bm Z_T$, define $\bm v \in \bm \Gamma_T$ such that $v_{xx}(x,y,z) = \int_{0}^x q_x ds$. The components $v_{yy}$ and $v_{zz}$ are similarly defined, and the off-diagonal components (e.g. $v_{xy}$) are defined as zero. Clearly, it holds that $\div \bm v = \bm q$.

Second, suppose that $\bm \sigma \in \bm \Phi_T$ such that $\curl \curl^{\mathsf T} \bm \sigma = 0$. Then there exists a vector-valued polynomial $\bm a$ such that
$\curl^{\mathsf T}  \bm \sigma = \grad \bm a$. Since $\curl^{\mathsf T}  \bm \sigma $ is a traceless matrix, therefore $\grad \bm a$ is traceless, which implies $ \div\bm a = 0$. Therefore there exists a vector-valued polynomial $\bm u$ such that $\bm a = \curl \bm u$. It then follows from \eqref{eq:curlTsymgrad} that 
$$\curl^{\mathsf T} \bm \sigma = \grad \curl \bm u =2 \curl^{\mathsf T} \sym \grad \bm u.$$
By the local exactness of the gradgrad complex, there exists a polynomial $\phi$ such that $\bm \sigma - 2\sym \grad \bm u = \grad \grad \phi$, therefore $\bm \sigma = \sym \grad(2\bm u + \grad \phi).$ A combination of these two results and the dimension counting (admitted here, and will be shown in \eqref{eq:dimcounting-elasticity}, 
implies the exactness.
% Second, suppose that $\bm v \in \bm \Gamma_T$ such that $\div \bm v = 0$. Then there exists a matrix-valued polynomial $\bm \sigma$ such that $\bm v = \curl \bm \sigma$. Since $\
\end{proof}
\subsection{\texorpdfstring{$\bm H^{1}$}{H1} conforming space}
This subsection considers the construction of an $\bm H^1$ conforming finite element vector-valued space $\bm X_h$ on $\mathcal T_h$. For this element, the shape function space on $T$ is as follows:
\begin{equation}
    \label{eq:shapefunc-H1R3}
\begin{bmatrix} u_x \\ u_y \\ u_z \end{bmatrix} \in
\begin{bmatrix}
    \cQ_{k,k+1,k+1} \\ \cQ_{k+1,k,k+1} \\\cQ_{k+1,k+1,k} \\ 
\end{bmatrix} := \bm X_T.
\end{equation}

% \paragraph{Degrees of Freedom on $u_x$}
The degrees of freedom of $u_x$ are defined as follows: 
\begin{enumerate}
\item The function value and the following partial derivatives of $u_x$, at each vertex $\bm x$ of $T$, 
\begin{equation}
\label{eq:dof-H1V2-v} u_x(\bm x), \py u_x(\bm x), \pz u_x(\bm x), \pyz u_x(\bm x).\end{equation}
\item The moments of function itself and partial derivatives of $u_x$ on each edge $e_x$ of $T$, 
\begin{equation}
\label{eq:dof-H1V2-ex}
\int_{e_x} u_xp~dx, \int_{e_x} \py  u_xp~dx, \int_{e_x} \pz  u_xp~dx , \int_{e_x} \pyz u_xp ~dx \for p \in \cQ_{k-2}(x).
\end{equation}
\item The moments of function itself and partial derivatives of $u_x$ on each edge $e_y,e_z$ of $T$, 
\begin{equation}
\label{eq:dof-H1V2-ey}
\int_{e_y} u_xp~dx, \int_{e_y} \pz u_xp~dx, \for p \in \cQ_{k-3}(y) 
\end{equation}
\begin{equation}
\label{eq:dof-H1V2-ez}
    \int_{e_z} u_xp~dx,\int_{e_z} \py u_xp~dx,  \for p \in \cQ_{k-3}(z).
\end{equation}
\item The moments of $u_x$ on each face $F_{yz}$, 
\begin{equation}
\label{eq:dof-H1V2-fyz}
\int_{F_{yz}} u_xp ~dydz \for p \in \cQ_{k-3,k-3}(y,z).
\end{equation}
\item The moments of function itself and partial derivatives of $u_x$ on $F_{xy}$ and $F_{xz}$ of $T$,
\begin{equation}
\label{eq:dof-H1V2-fxy}
    \int_{F_{xy}} u_xp ~dxdy,  \int_{F_{xy}} \pz u_xp ~dxdy \for p \in \cQ_{k-2,k-3}(x,y) ,
\end{equation}    
\begin{equation}
\label{eq:dof-H1V2-fxz}\int_{F_{xz}} u_xp ~dxdz,  \int_{F_{xz}} \py u_xp ~dxdz \for p \in \cQ_{k-2,k-3}(x,z).
\end{equation}
\item The moment of $u_x$ inside the element $T$,
\begin{equation}
\label{eq:dof-H1V2-t}
\int_{T} u_xp ~dxdydz \for p \in \cQ_{k-2,k-3,k-3}(x,y,z).
\end{equation}
\end{enumerate}

The degrees of freedom of the remaining components ($u_y$ and $u_z$) can be similarly defined, by a cyclic permutation on the index.
\begin{proposition}
    When $k \ge 2$, then the degrees of freedom defined above are unisolvent with respect to the shape function space \eqref{eq:shapefunc-H1R3}, and the resulting finite element space $\bm X_h$ is $\bm H^1$ conforming.
\end{proposition}

\begin{proof}
The number of degrees of freedom defined for $u_x$ is 
$$ 4 + 16(k-1) + 16(k-2) + 2(k-2)^2 + 4(k-1)(k-2) + (k-1)(k-2)^2 = (k+1)(k+2)^2,$$
which is equal to the dimension of $\cQ_{k,k+1,k+1}$.
It then suffices to show that if $u_x \in \cQ_{k,k+1,k+1}$ vanishes at the degrees of freedom defined above, then $u_x = 0$. 
It follows from \eqref{eq:dof-H1V2-v} and \eqref{eq:dof-H1V2-ex} that $$u_{x},\frac{\partial}{\partial y}u_{x},\frac{\partial}{\partial z}u_{x}, \frac{\partial^2}{\partial y\partial z}u_{x}$$  vanish at all vertices and edges $e_{x}$. 
Then it follows from \eqref{eq:dof-H1V2-ey} and \eqref{eq:dof-H1V2-ez} that $u_{x},\frac{\partial}{\partial z}u_{x}$ vanish on edges $e_{y}$ and $u_{x},\frac{\partial}{\partial y}u_{x}$ vanish on edges $e_{z}$. 
Therefore, on any face, say $F_{yz}$, the restriction $u_{x}|_{F_{yz}}$ can be expressed as 
$$u_{x}|_{F_{yz}} = y^2(1-y)^2z^2(1-z)^2u_1, u_1 \in \cQ_{k-3,k-3}(y,z).$$
Due to \eqref{eq:dof-H1V2-fyz}, it is easy to see that $u_{x}$ vanishes on $F_{yz}$. Similarly, $u_{x}$ vanishes on all faces and $\frac{\partial}{\partial z}u_{x}$ vanishes on face $F_{xy}$, $\frac{\partial}{\partial y}u_{x}$ vanishes on face $F_{xz}$, by  \eqref{eq:dof-H1V2-fxy} and \eqref{eq:dof-H1V2-fxz}. Hence $u_{x}=x(1-x)y^2(1-y)^2z^2(1-z)^2u_2$ for some $u_2 \in \cQ_{k-2,k-3,k-3}$ and consequently is zero, by \eqref{eq:dof-H1V2-t}.

\end{proof}

The dimension of $\bm X_h$ is 
\begin{equation}
\dim \bm X_h = 12 \mathscr{V} + [4(k-1) + 4(k-2) ] \mathscr{E} + [(k-2)^2+ 4(k-1)(k-2)] \mathscr{F} + 3(k-1)(k-2)^2 \mathscr{T}.
\end{equation}

\begin{remark}
As it will be shown later, $\bm X_h$ is actually $\bm H^1(\curl)$ conforming.
\end{remark}
\subsection{\texorpdfstring{$\bm H(\operatorname{curl} \operatorname{curl}^{\mathsf{T}} ; \mathbb{S})$}{H(curl curl;S)} conforming space}
The discrete 
$\bm H(\curl \curl^{\mathsf{T}};\bS)$ conforming space $\bm \Phi_h$ (of degree $k$) is taken exactly as the proposed $\bm H(\curl;\bS)$ conforming space $\bm \Sigma_h$ of degree $k+1$, see \Cref{sec:HcurlS}.
The following lemma shows that the proposed $\bm \Sigma_h$ has $\bm H(\curl \curl^{\mathsf{T}} ; \bS)$ regularity.
\begin{lemma}[$\bm H(\curl \curl^{\mathsf{T}};\bS)$ conformity]
     If a symmetric matrix-valued piecewise polynomial $\bm \sigma$ satisfies that
    \begin{itemize}
    % \item $\sigma_{xy}$ continuous on $\mathcal T_h$.
    % \item $\sigma_{xx}$ and $\frac{\partial \sigma_{xx}}{\partial \bm n}$ is single-valued across $F_{xy}$ and $F_{xz}$ of $\mathcal T_h$. Similar condition holds for $\sigma_{yy}$ and $\sigma_{zz}$. 
    \item The diagonal component $\sigma_{xx}$ is single-valued across $e_x$ of $\mathcal T_h$, $\sigma_{xx}$ and $\frac{\partial \sigma_{xx}}{\partial \bm n}$ is single-valued across $F_{xy}$ and $F_{xz}$ of $\mathcal T_h$. A corresponding condition holds for $\sigma_{yy}$ and $\sigma_{zz}$. 
    \item The off-diagonal component $\sigma_{xy}$ is continuous on $\mathcal T_h$, $\pz \sigma_{xy}$ is single-valued across $F_{xy}$ of $\mathcal T_h$. A corresponding condition holds for $\sigma_{xz}$ and $\sigma_{yz}$.
    \end{itemize}
    then $\bm \sigma \in \bm H(\curl \curl^{\mathsf{T}}; \bS)$.
\end{lemma}
\begin{proof}
Set $\bm v = \curl \curl^{\mathsf{T}} \bm \sigma$, then a direct calculation yields that 
$$v_{xx} = \frac{\partial^2 \sigma_{yy}}{\partial z^2} + \frac{\partial^2 \sigma_{zz}}{\partial y^2} - 2\frac{\partial^2 \sigma_{yz}}{\partial y\partial z},$$
$$ v_{xy} =  \frac{\partial^2 \sigma_{xy}}{\partial z^2} + \frac{\partial ^2 \sigma_{zz}}{\partial x \partial y} - \frac{\partial^2 \sigma_{xz}}{\partial y \partial z} - \frac{\partial^2 \sigma_{yz}}{\partial x \partial z}.$$

% Then it follows from straightforward calculation.
By assumption it holds that $$\int_{\mathcal T} \frac{\partial^2 \sigma_{yy}}{\partial z^2} p = \int_{\mathcal T}  \frac{\partial^2 p}{\partial z^2} \sigma_{yy} + \int_{\mathcal F_{xy}} [\![ \frac{\partial \sigma_{yy}}{\partial z} ]\!] p + \int_{\mathcal F_{xy}} [\![ \sigma_{yy}]\!] \frac{\partial p}{\partial z} =  \int_{\mathcal T}\frac{\partial^2 p}{\partial z^2} \sigma_{yy}.$$
$$\int_{\mathcal T} \frac{\partial^2 \sigma_{yz}}{\partial y\partial z} p = \int_{\mathcal T} \frac{\partial^2 p}{\partial y\partial z} \sigma_{yz} + \int_{\mathcal F_{xy}} [\![\sigma_{yz}]\!] p + \int_{\mathcal F_{xz}} [\![\sigma_{yz}]\!] p + \int_{\mathcal E_{x}} [\![\sigma_{yz}]\!] p = \int_{\mathcal T} \frac{\partial^2 p}{\partial y\partial z} \sigma_{yz}.$$

The other components are calculated similarly. As a result, it can be concluded that $\curl \curl^{\mathsf T} \bm \sigma \in \bm L^2$.
\end{proof}

From this lemma and the proof of \Cref{sec:HcurlS}, one can easily derive the following proposition.

\begin{proposition}
When $k \ge 3$, the space $\bm \Phi_h$ is $\bm H(\curl \curl^{\mathsf{T}}; \bS)$ conforming.
\end{proposition}

\subsection{\texorpdfstring{$\bm H(\operatorname{div}; \mathbb{S})$}{H(div;S)} conforming space}

The $\bm H(\div;\bS)$ conforming space $\bm \Gamma_h$ is taken as the $\bm H(\div \div; \bS)$ conforming space in \cite{hu2022new}.
The shape function spaces of $\bm \Gamma_h$ on $T$ is taken as 
\begin{equation}
    \label{eq:shapefunc-HdivdivS}
    \begin{bmatrix}
    \sigma_{xx} & \sigma_{xy} & \sigma_{xz} \\
    \sigma_{yx} & \sigma_{yy} & \sigma_{yz} \\
    \sigma_{zx} & \sigma_{zy} & \sigma_{zz}
    \end{bmatrix}
    \in
    \begin{bmatrix}
        \cQ_{k+1,k-1,k-1} & \cQ_{k,k,k-1} & \cQ_{k,k-1,k}\\
        \cQ_{k,k,k-1} & \cQ_{k-1,k+1,k-1} & \cQ_{k-1,k,k} \\
        \cQ_{k,k-1,k} & \cQ_{k-1,k,k} & \cQ_{k-1,k-1,k+1}
    \end{bmatrix}  =: \bm \Gamma_T.
    \end{equation}

    \paragraph{The degrees of freedom on $\sigma_{xx}$} is defined as follows:
\begin{enumerate}
\item The moments of $\sigma_{xx}$ and $\px \sigma_{xx}$ on the face $F_{yz}$,
\begin{equation}
    \label{eq:dof-HdivdivS-fyz}
\int_{F_{yz}} \sigma_{xx} p~dydz  ~~~\int_{F_{yz}} \px\sigma_{xx} p~dydz \for p\in \cQ_{k-1,k-1}(y,z).
\end{equation}
\item The moments of $\sigma_{xx}$ inside the element $T$,
\begin{equation}
    \label{eq:dof-HdivdivS-t}
\int_T \sigma_{xx} \for p \in \cQ_{k-3,k-1,k-1}(x,y,z).
\end{equation}
\end{enumerate}

\paragraph{The degrees of freedom on $\sigma_{xy}$} is defined as follows:
\begin{enumerate}
\item The moments of $\sigma_{xy}$ on each edge $e_z$ of $T$, 
\begin{equation}
    \label{eq:dof-HdivdivS2-ez}
\int_{e_z} \sigma_{xy} p ~dz \for p \in \cQ_{k-1}(z).
\end{equation}
\item The moments of $\sigma_{xy}$ on each face $F_{xz}$ and $F_{yz}$ of $T$,
\begin{equation}
    \label{eq:dof-HdivdivS2-f}
\int_{F_{xz}} \sigma_{xy} p~dxdz \for p \in \cQ_{k-2,k-1}(x,z), \int_{F_{yz}} \sigma_{xy} p~dydz \for p \in \cQ_{k-2,k-1}(y,z).
\end{equation}
\item The moments of $\sigma_{xy}$ inside the element $T$,
\begin{equation}
    \label{eq:dof-HdivdivS2-t}
\int_T \sigma_{xy} p~dxdydz \for p \in \cQ_{k-2,k-2,k-1}(x,y,z).
\end{equation}

\end{enumerate}

The proof of unisolvency and $\bm H(\div;\bS)$ conformity is classical, and actually shown in \cite{hu2022new}.

The dimension of $\bm \Theta_h$ is 
$$k \mathscr{E} +  [2k^2 + 2k(k-1)] \mathscr{F} +  [3k^2(k-2) + 3k(k-1)^2] \mathscr{T}.$$

\subsection{\texorpdfstring{$\bm L^2$}{L2} conforming space}

This section constructs an $\bm L^2$ finite element space $\bm Z_h$. For this element, the shape function space on $T$ is 
\begin{equation}
    \label{eq:shapefunc-L2R3}
\begin{bmatrix} q_x \\ q_y \\ q_z \end{bmatrix} \in
\begin{bmatrix}
    \cQ_{k,k-1,k-1} \\ \cQ_{k-1,k,k-1} \\\cQ_{k-1,k-1,k} \\ 
\end{bmatrix} =: \bm Z_T.
\end{equation}
The degrees of freedom of $q_x$ are defined as follows:
\begin{enumerate}
\item The moment of $q_x$ on each face $F_{yz}$ of $T$,
\begin{equation}
\int_{F_{yz}} q_x p~dydz \for p \in \cQ_{k-1,k-1}(y,z).
\end{equation}
\item The moment of $q_x$ inside $T$,
\begin{equation}
\int_T q_x p~dxdydz \for p \in \cQ_{k-2,k-1,k-1}(x,y,z).
\end{equation}
\end{enumerate}

The degrees of freedom of $q_y$ and $q_z$ can be similarly defined as those of $q_x$, by a cyclic permutation. The verification of the unisolvency of $\bm Z_h$ is straightforward and hence omitted.
The dimension of $\bm Z_h$ is
$$\dim \bm Z_h =  k^2 \mathscr{F} + 3(k-1)k^2 \mathscr{T}.$$

\subsection{Finite element complex and its exactness}
This section proves that the following finite element sequence
\begin{equation}\label{eq:elasticitycomplex-dis}
    %\label{eq:divdivcomplex-dis}
    \bm{\mathcal{RM}} \stackrel{\subseteq}{\longrightarrow} \bm X_h \stackrel{\operatorname{sym} \operatorname{grad}}{\longrightarrow} \bm{\Phi}_h \stackrel{ \operatorname{curl} \operatorname{curl}^{\mathsf{T}} }{\longrightarrow} \bm \Gamma_h \stackrel{\div}{\longrightarrow} \bm Z_h\longrightarrow 0,
    \end{equation}
is an exact complex.

Since there hold the following dimensions of these spaces $\bm X_h, \bm \Phi_h, \bm \Gamma_h, \bm Z_h$ defined in the previous subsections,

\begin{equation*}
    \begin{aligned}
        \dim  \bm {X}_h = &  12 \mathscr{V} + [4(k-1) + 4(k-2) ] \mathscr{E} + [(k-2)^2+ 4(k-1)(k-2)] \mathscr{F} + 3(k-1)(k-2)^2 \mathscr{T},\\
        \dim \bm{\Phi}_h = &[4k\mathscr{E} + 4k(k-2)\mathscr{F} + 3k(k-2)^2\mathscr{T}]\\
        &+[6\mathscr{V} + 4(k-1)\mathscr{E} + (k-2)\mathscr{E} + 2(k-1)^2\mathscr{F} + 2(k-1)(k-2)\mathscr{F} + 3(k-1)^2(k-2)\mathscr{T}],
    \\
            \dim \bm{\Gamma}_h =  &k \mathscr{E} +  [2k^2 + 2k(k-1)] \mathscr{F} +  [3k^2(k-2) + 3k(k-1)^2] \mathscr{T},
    \\
            \dim \bm Z_h = &k^2 \mathscr{F} + 3(k-1)k^2 \mathscr{T},
    \end{aligned}
    \end{equation*}

this leads to

\begin{equation}
    \label{eq:dimcounting-elasticity}
    \begin{split}
\dim \bm X_h - \dim \bm \Phi_h + \dim \bm \Gamma_h - \dim \bm Z_h = & 6(\mathscr{V} - \mathscr{E} + \mathscr{F} - \mathscr{T}) \\ = & 6  = \dim \bm{\mathcal{RM}}
    \end{split}
\end{equation}
by Euler's formula.

The following result holds, indicating that the discrete divergence operator is surjective.
\begin{proposition}
\label{prop:exactness-elasticity-surjective}
The discrete divergence operator:
$\div: \bm \Gamma_h \to \bm Z_h$ is surjective.
\end{proposition}
\begin{proof}
The argument is similar to that in \cite{hu2014simple}. 
Consider the symmetric matrix-valued piecewise polynomial function $\bm \sigma$ such that 
\begin{equation}
\sigma_{xx} = \int_{0}^x q_x ds,\,\, \sigma_{yy} = \int_{0}^y q_y ds, \,\,
\sigma_{zz} = \int_{0}^z q_z ds.
\end{equation}
Set $\sigma_{xy} = \sigma_{xz} = \sigma_{yz} = 0$. 
Clearly, it holds that $\bm \sigma |_T \in \bm \Phi_T$, it then suffices to check the continuity. By such a construction, $\tau_{xx}$ is continuous on each face $F_{yz}$. It directly follows from the facts that $\px{\sigma_{xx}} = q_x$ and $q_x$ is continuous when crossing $F_{yz}$, which completes the proof.
\end{proof}

The following proposition shows that the kernel of the discrete $\curl \curl^{\mathsf T}$ operator is just the image of $\sym \grad$.
\begin{proposition}
\label{prop:exactness-elasticity-first}
    % $\bm H^{1}\left(\Omega ; \bR^3\right) \stackrel{\operatorname{sym} \operatorname{grad}}{\longrightarrow}  \bm H(\operatorname{curl} \operatorname{curl}^{\mathsf{T}}, \Omega ; \mathbb{S}) \stackrel{\operatorname{curl} \operatorname{curl}^{\mathsf{T}}}{\longrightarrow} \bm H(\div div,\Omega; \bS)$ is exact.
    If $ \bm \sigma \in \bm \Phi_h$ satisfies that 
    $\curl \curl^{\mathsf{T}} \bm \sigma = 0$, then there exists $\bm v \in \bm X_h$ such that $\bm\sigma = \sym \grad \bm v$.
\end{proposition}

\begin{proof}
It follows from \eqref{eq:complex-elasticity-C} that $\curl \curl^{\mathsf{T}} \bm \sigma = 0$ implies that there exists a function $\bm u \in \bm H^1(\Omega)$ such that 
$$(\sym \grad {\bm u}) = \bm \sigma.$$
It follows from the local discrete complex that each ${\bm u}|_K$ is a polynomial and $\bm{u}|_{K}\in \bm{X}_{K}$, for each element $K$. Since ${\bm u} \in \bm H^1(\Omega)$, $\bm{u}$ is continuous when crossing the internal faces. Additionally, since $\bm \Phi_h$ is $\bm H(\curl)$ conforming by Proposition~\ref{prop:uni-HcurlS} , which implies $$\curl \sym \grad \bm u \in \bm L^2(\Omega; \mathbb M).$$
By \eqref{eq:curlsymgrad} $$\curl \sym \grad \bm u = \frac{1}{2} (\grad \curl \bm u)^{\mathsf T} ,$$ it shows that $\curl u \in \bm H^1(\Omega)$. Since $\curl \bm u$ is a piecewise polynomial, it then implies that $\curl \bm u $ is continuous. 

It follows from both 
$[\curl \bm u]_{x} = \py u_z - \pz u_y$  and $[\sym \grad u]_{yz} = \py u_z + \pz u_y$ are continuous that so are $\py u_z$ and 
$\pz u_y$. Similarly, it holds that $\py u_x, \pz u_x$, $\px u_y$, $\px u_z$ are continuous. 

It remains to show that $\pyz u_x$ are single-valued at each vertex of $\bm x$, which follows from a simple calculation
$$\pyz u_x = \frac{1}{2}[\py \sigma_{xz} + \pz \sigma_{xy} - \px \sigma_{yz}].$$ Since the three terms on the right hand side are single-valued at each vertex, by the degrees of freedom, $\pyz u_x$ is also single-valued.
\end{proof}

% \begin{lemma}
% If $\bm v $ and $\bm \sigma$ satisfies that $\bm \sigma = \sym \grad \bm v$ is a polynomial, then $\bm v$ itself is also a polynomial.
% \end{lemma}

\begin{theorem}
    When $k \ge 2$, the finite element sequence \eqref{eq:elasticitycomplex-dis} is an exact complex.
\end{theorem}
\begin{proof}
A combination of \Cref{prop:exactness-elasticity-surjective}, \Cref{prop:exactness-elasticity-first}, and the dimension counting \eqref{eq:dimcounting-elasticity} completes the proof.
\end{proof}
\subsection{A finite element complex with reduced regularity}

This section considers a finite element complex, which reduces some additional regularity. The goal of this subsection is to construct

\begin{equation}\label{eq:elasticitycomplex-dis2}
    %\label{eq:divdivcomplex-dis}
    \bm{\mathcal{RM}} \stackrel{\subseteq}{\longrightarrow} \bm X_h \stackrel{\operatorname{sym} \operatorname{grad}}{\longrightarrow} \bm{\Phi}_h \stackrel{ \operatorname{curl} \operatorname{curl}^{\mathsf{T}} }{\longrightarrow} \widetilde{\bm \Gamma}_h \stackrel{\div}{\longrightarrow} \widetilde{\bm Z}_h\longrightarrow 0,
    \end{equation}
\subsubsection{\texorpdfstring{$\bm H(\operatorname{div}; \mathbb{S})$}{H(div;S)} conforming space with reduced regularity}

This subsubsection introduces an $\bm H(\operatorname{div}; \mathbb{S})$ conforming element space $\bm{\widetilde{\Gamma}}_h$ from \cite{hu2014simple}, with reduced regularity. The shape function space on $T$ is taken as
\begin{equation}
    \label{eq:shapefunc-HdivS-new}
    \begin{bmatrix}
    \sigma_{xx} & \sigma_{xy} & \sigma_{xz} \\
    \sigma_{yx} & \sigma_{yy} & \sigma_{yz} \\
    \sigma_{zx} & \sigma_{zy} & \sigma_{zz}
    \end{bmatrix}
    \in
    \begin{bmatrix}
        \cQ_{k+1,k-1,k-1} & \cQ_{k,k,k-1} & \cQ_{k,k-1,k}\\
        \cQ_{k,k,k-1} & \cQ_{k-1,k+1,k-1} & \cQ_{k-1,k,k} \\
        \cQ_{k,k-1,k} & \cQ_{k-1,k,k} & \cQ_{k-1,k-1,k+1}
    \end{bmatrix} = \bm \Gamma_T.
    \end{equation}

    \paragraph{The degrees of freedom on $\sigma_{xx}$} is defined as follows:
\begin{enumerate}
\item The moments of $\sigma_{xx}$ on each face $F_{yz}$ of $T$,
\begin{equation}
\int_{F_{yz}} \sigma_{xx} p~dydz \for p\in \cQ_{k-1,k-1}(y,z).
\end{equation}
\item The moments of $\sigma_{xx}$ inside $T$,
\begin{equation}
\int_T \sigma_{xx} \for p \in \cQ_{k-1,k-1,k-1}(x,y,z).
\end{equation}
\end{enumerate}

\paragraph{The degrees of freedom on $\sigma_{xy}$} is defined as 
\begin{enumerate}
\item The moments of $\sigma_{xy}$ on each edge $e_z$,
\begin{equation}
\int_{e_z} \sigma_{xy} p ~dz \for p \in \cQ_{k-1}(z).
\end{equation}
\item The moments of $\sigma_{xy}$ on each face $F_{xz}$ and $F_{yz}$ of $T$, normal to $y$- and $x$-direction,
\begin{equation}
\int_{F_{xz}} \sigma_{xy} p~dxdz \for p \in \cQ_{k-2,k-1}(x,z),
\end{equation}
\begin{equation}
    \int_{F_{yz}} \sigma_{xy} p~dydz \for p \in \cQ_{k-2,k-1}(y,z).
\end{equation}
\item The moments of $\sigma_{xy}$ inside $T$,
\begin{equation}
\int_T \sigma_{xy} p~dxdydz \for p \in \cQ_{k-2,k-2,k-1}(x,y,z).
\end{equation}
\end{enumerate}

\begin{proposition}
When $k \ge 2$, the degrees of freedom above are unisolvent with respect to the shape function space \eqref{eq:shapefunc-HdivS-new}, and the resulting finite element space $\bm {\widetilde{\Gamma}}_h$ is $\bm H(\div;\mathbb S)$ conforming.
\end{proposition}
\begin{proof}
% \textbf{Unisolvence of $\sigma_{xx}$}
% Omitted.

% \textbf{Unisolvence of $\sigma_{xy}$}
% Same as that in [REF].
It follows from a similar argument as that of \cite{hu2014simple}.

\end{proof}
The dimension of the $\bm{\widetilde{\Gamma}}_h$ is 
$$ \bm{\widetilde{\Gamma}}_h = k \mathscr{E} +  [k^2 + 2k(k-1)] \mathscr{F} +  [3k^3 + 3k(k-1)^2] \mathscr{T}.$$
\subsubsection{A finite element complex with regularity and its exactness}
Take $\bm{\widetilde{Z}}_h$ as the vector-valued discontinuous element space, and the following relationship holds, 
% \begin{equation}
%     = \begin{bmatrix}
%         \mathcal{DG}_{k,k-1,k-1} \\ \mathcal{DG}_{k-1,k,k-1} \\ \mathcal{DG}_{k-1,k-1,k}\\ 
%     \end{bmatrix}
% \end{equation}

% Since 

\begin{equation*}
    \begin{aligned}
        \dim  \bm {X}_h = &  12 \mathscr{V} + [4(k-1) + 4(k-2) ] \mathscr{E} + [(k-2)^2+ 4(k-1)(k-2)] \mathscr{F} + 3(k-1)(k-2)^2 \mathscr{T}.\\
        \dim \bm{\Phi}_h = &[4k\mathscr{E} + 4k(k-2)\mathscr{F} + 3k(k-2)^2\mathscr{T}]\\
        &+[6\mathscr{V} + 4(k-1)\mathscr{E} + (k-2)\mathscr{E} + 2(k-1)^2\mathscr{F} + 2(k-1)(k-2)\mathscr{F} + 3(k-1)^2(k-2)\mathscr{T}]
    \\
            \dim \bm{\widetilde{\Gamma}}_h =  &k \mathscr{E} +  [k^2 + 2k(k-1)] \mathscr{F} +  [3k^3 + 3k(k-1)^2] \mathscr{T}.
    \\
            \dim \bm {\widetilde{Z}}_h = & 3(k+1)k^2 \mathscr{T}.
    \end{aligned}
    \end{equation*}

    A calculation yields that

\begin{equation}
\label{eq:dimcounting-elanew}
    \begin{split}
\dim \bm X_h - \dim \bm \Phi_h + \dim \bm{\widetilde{\Gamma}}_h - \dim \bm {\widetilde{Z}}_h  = & 6(\mathscr{V} - \mathscr{E} + \mathscr{F} - \mathscr{T}) \\ = & 6  = \dim \bm{\mathcal{RM}}.
    \end{split}
\end{equation}

    \begin{proposition}
    \label{prop:exactness-elasticitynew-surjective}
        When $k \ge 2$, it holds that $\div \bm{\widetilde{\Gamma}}_h = \bm {\widetilde{Z}}_h$.
        \end{proposition}
        \begin{proof}
        The proof is similar to that of \Cref{prop:exactness-elasticity-surjective}.
        \end{proof}
        
        \begin{theorem}
When $k \ge 2$, the finite element sequence \eqref{eq:elasticitycomplex-dis2} is an exact complex.
        \end{theorem}
        \begin{proof}
A combination of dimension counting \eqref{eq:dimcounting-elanew}, \Cref{prop:exactness-elasticitynew-surjective,prop:exactness-elasticity-first} completes the proof.
        \end{proof}

\bibliographystyle{plain}
\bibliography{ref}

\begin{thebibliography}{10}

\bibitem{2008ArnoldAwanouWinther}
D.~N. Arnold, G.~Awanou, and R.~Winther.
\newblock Finite elements for symmetric tensors in three dimensions.
\newblock {\em Math. Comp.}, 77(263):1229--1251, 2008.

\bibitem{2021ArnoldHu}
D.~N. Arnold and K.~Hu.
\newblock Complexes from complexes.
\newblock {\em Found. Comput. Math.}, 21(6):1739--1774, 2021.

\bibitem{2002ArnoldWinther}
D.~N. Arnold and R.~Winther.
\newblock Mixed finite elements for elasticity.
\newblock {\em Numer. Math.}, 92(3):401--419, 2002.

\bibitem{2018Arnold}
Douglas~N. Arnold.
\newblock {\em Finite element exterior calculus}, volume~93 of {\em CBMS-NSF
  Regional Conference Series in Applied Mathematics}.
\newblock Society for Industrial and Applied Mathematics (SIAM), Philadelphia,
  PA, 2018.

\bibitem{2006ArnoldFalkWinther}
Douglas~N. Arnold, Richard~S. Falk, and Ragnar Winther.
\newblock Finite element exterior calculus, homological techniques, and
  applications.
\newblock {\em Acta Numer.}, 15:1--155, 2006.

\bibitem{2010ArnoldFalkWinther}
Douglas~N. Arnold, Richard~S. Falk, and Ragnar Winther.
\newblock Finite element exterior calculus: from {H}odge theory to numerical
  stability.
\newblock {\em Bull. Amer. Math. Soc. (N.S.)}, 47(2):281--354, 2010.

\bibitem{2013BoffiBrezziFortin}
Daniele Boffi, Franco Brezzi, Michel Fortin, et~al.
\newblock {\em Mixed finite element methods and applications}, volume~44.
\newblock Springer, 2013.

\bibitem{brenner2007mathematical}
Susanne Brenner and Ridgway Scott.
\newblock {\em The mathematical theory of finite element methods}, volume~15.
\newblock Springer Science \& Business Media, 2007.

\bibitem{2020ChenHuang3D}
L.~{Chen} and X.~{Huang}.
\newblock Finite elements for {${\rm div\,div}$} conforming symmetric tensors
  in three dimensions.
\newblock {\em Math. Comp.}, 91(335):1107--1142, 2022.

\bibitem{chen2022finite}
Long Chen and Xuehai Huang.
\newblock A finite element elasticity complex in three dimensions.
\newblock {\em Mathematics of Computation}, 91(337):2095--2127, 2022.

\bibitem{2018ChristianseHuHu}
S.~H. Christiansen, J.~Hu, and K.~Hu.
\newblock Nodal finite element de {R}ham complexes.
\newblock {\em Numer. Math.}, 139(2):411--446, 2018.

\bibitem{2020Christiansen}
Snorre~H Christiansen, Jay Gopalakrishnan, Johnny Guzm{\'a}n, and Kaibo Hu.
\newblock A discrete elasticity complex on three-dimensional {A}lfeld splits.
\newblock {\em arXiv preprint arXiv:2009.07744}, 2020.

\bibitem{christiansen2022finite}
Snorre~H Christiansen and Kaibo Hu.
\newblock Finite element systems for vector bundles: elasticity and curvature.
\newblock {\em Foundations of Computational Mathematics}, pages 1--52, 2022.

\bibitem{MR3352360}
Jun Hu.
\newblock Finite element approximations of symmetric tensors on simplicial
  grids in {$\Bbb R^n$}: the higher order case.
\newblock {\em J. Comput. Math.}, 33(3):283--296, 2015.

\bibitem{2021HuLiang}
Jun Hu and Yizhou Liang.
\newblock Conforming discrete gradgrad-complexes in three dimensions.
\newblock {\em Mathematics of Computation}, 90(330):1637--1662, 2021.

\bibitem{hu2021divdiv1}
Jun Hu, Yizhou Liang, and Rui Ma.
\newblock Conforming finite element divdiv complexes and the application for
  the linearized einstein--bianchi system.
\newblock {\em SIAM Journal on Numerical Analysis}, 60(3):1307--1330, 2022.

\bibitem{hu2022new}
Jun Hu, Yizhou Liang, Rui Ma, and Min Zhang.
\newblock New conforming finite element divdiv complexes in three dimensions.
\newblock {\em arXiv preprint arXiv:2204.07895}, 2022.

\bibitem{hu2014simple}
Jun Hu, Hongying Man, and Shangyou Zhang.
\newblock A simple conforming mixed finite element for linear elasticity on
  rectangular grids in any space dimension.
\newblock {\em Journal of Scientific Computing}, 58(2):367--379, 2014.

\bibitem{MR3301063}
Jun Hu and ShangYou Zhang.
\newblock A family of symmetric mixed finite elements for linear elasticity on
  tetrahedral grids.
\newblock {\em Sci. China Math.}, 58(2):297--307, 2015.

\bibitem{MR3529252}
Jun Hu and Shangyou Zhang.
\newblock Finite element approximations of symmetric tensors on simplicial
  grids in {$\Bbb{R}^n$}: the lower order case.
\newblock {\em Math. Models Methods Appl. Sci.}, 26(9):1649--1669, 2016.

\bibitem{dirk2020}
Dirk Pauly and Walter Zulehner.
\newblock The div{D}iv-complex and applications to biharmonic equations.
\newblock {\em Appl. Anal.}, 99(9):1579--1630, 2020.

\bibitem{quenneville2015new}
V.~Quenneville-Belair.
\newblock A {N}ew {A}pproach to {F}inite {E}lement {S}imulations of {G}eneral
  {R}elativity.
\newblock page 113, 2015.
\newblock Thesis (Ph.D.)--University of Minnesota.

\bibitem{schmit1968finite}
LA~Schmit, FK~Bogner, and RL~Fox.
\newblock Finite deflection structural analysis using plate and shell
  discreteelements.
\newblock {\em AIAA Journal}, 6(5):781--791, 1968.

\end{thebibliography}

\end{document}